\documentclass[12pt, twoside]{article}
\usepackage{amsmath,amsthm,amssymb}
\usepackage{times}
\usepackage{enumerate}

\def\titlerunning#1{\gdef\titrun{#1}}
\makeatletter
\def\author#1{\gdef\autrun{\def\and{\unskip, }#1}\gdef\@author{#1}}
\def\address#1{{\def\and{\\\hspace*{18pt}}\renewcommand{\thefootnote}{}%
\footnote {#1}}%
\markboth{\autrun}{\titrun}}
\makeatother
\def\email#1{e-mail: #1}
\def\subjclass#1{{\renewcommand{\thefootnote}{}%
\footnote{\emph{Mathematics Subject Classification (2010):} #1}}}
\def\keywords#1{\par\medskip
\noindent\textbf{Keywords.} #1}

\newtheorem{theorem}{Theorem}[section]
\newtheorem{lemma}[theorem]{Lemma}
\newtheorem{definition}[theorem]{Definition}
\newtheorem{proposition}[theorem]{Proposition}

\newtheorem{corollary}[theorem]{Corollary}

\newcommand{\Proof}{\begin{proof}}
\newcommand{\End}{\end{proof}}

\numberwithin{equation}{section}

\begin{document}


\baselineskip=15pt


\titlerunning{Variational principle for contact Hamiltonian systems and its applications}

\title{Variational principle for contact Hamiltonian systems and its applications}

\author{Kaizhi Wang \and Lin Wang \and Jun Yan}

\date{}

\maketitle

\address{Kaizhi Wang: School of Mathematical Sciences, Shanghai Jiao Tong University, Shanghai 200240, China; \email{kzwang@sjtu.edu.cn}
\and
Lin Wang: Yau Mathematical Sciences Center, Tsinghua University, Beijing 100084, China; \email{lwang@math.tsinghua.edu.cn}
\and Jun Yan: School of Mathematical Sciences, Fudan University and Shanghai Key Laboratory for Contemporary Applied Mathematics, Shanghai 200433, China;
\email{yanjun@fudan.edu.cn}}

\begin{abstract}
 In \cite{WWY}, the authors provided an implicit variational principle for the contact Hamilton's equations

\begin{align*}
\left\{
\begin{array}{l}
\dot{x}=\frac{\partial H}{\partial p}(x,u,p),\\
\dot{p}=-\frac{\partial H}{\partial x}(x,u,p)-\frac{\partial H}{\partial u}(x,u,p)p,\quad (x,p,u)\in T^*M\times\mathbf{R},\\
\dot{u}=\frac{\partial H}{\partial p}(x,u,p)\cdot p-H(x,u,p),
\end{array}
\right.
\end{align*}
where $M$ is a closed, connected and smooth manifold and $H=H(x,u,p)$ is strictly convex, superlinear  in $p$ and Lipschitz in $u$. In the present paper, we focus on two applications of the variational principle: 1. We provide a  representation formula for the solution semigroup of the evolutionary equation

\[
w_t(x,t)+H(x,w(x,t),w_x(x,t))=0;
\]
2. We study the ergodic problem of the stationary equation via the solution semigroup. More precisely, we find pairs $(u,c)$ with $u\in C(M,\mathbf{R})$ and $c\in\mathbf{R}$ which, in the viscosity sense, satisfy the stationary partial differential equation

\[
H(x,u(x),u_x(x))=c.
\]

\noindent \textbf{R\'esum\'e}

Dans \cite{WWY}, les auteurs ont fourni un principe variationnel implicite pour le contact des \'equations de Hamilton

\begin{align*}
\left\{
\begin{array}{l}
\dot{x}=\frac{\partial H}{\partial p}(x,u,p),\\
\dot{p}=-\frac{\partial H}{\partial x}(x,u,p)-\frac{\partial H}{\partial u}(x,u,p)p,\quad (x,p,u)\in T^*M\times\mathbf{R},\\
\dot{u}=\frac{\partial H}{\partial p}(x,u,p)\cdot p-H(x,u,p),
\end{array}
\right.
\end{align*}
o$\mathrm{\grave{u}}$ M est une vari\'et\'e ferm\'e, connexe et lisse et $H =H(x,u,p)$ est strictement convexe, superlineaire en $p$ et Lipschitz en $u$. Dans cette papier, on se concentre sur deux applications du principe variationnel: 1. On fournit une formule de repr\'esentation pour le demi-groupe de solution de l'\'equation d\'evolution:

\[
w_t(x,t)+H(x,w(x,t),w_x(x,t))=0;
\]
2. On \'etudie le prob$\mathrm{\grave{e}}$me ergodique de l'\'equation stationnaire via le demi-groupe de solution. Plus pr\'ecis\'ement, nous trouvons des paires $(u,c)$ avec $u\in C(M,\mathbf{R})$ et $c\in\mathbf{R}$ qui, au sens de la viscosit\'e,
satisfont l'\'equation stationnaire aux d\'eriv\'ees partielles

\[
H(x,u(x),u_x(x))=c.
\]

\end{abstract}

\tableofcontents


\section{Introduction}
\setcounter{equation}{0}

Let $M$ be a closed (i.e., compact, without boundary), connected and smooth manifold. We choose, once and for all, a $C^\infty$ Riemannian metric $g$ on $M$.  Let $H:T^*M\times\mathbf{R}\to \mathbf{R}$ be a $C^3$  function called a contact Hamiltonian.

The aim of this paper is threefold: \emph{1)} to study more interesting properties of the implicit action function introduced in the implicit variational principle established in \cite{WWY} for the contact Hamilton's equations
\begin{align}\label{che}
\left\{
\begin{array}{l}
\dot{x}=\frac{\partial H}{\partial p}(x,u,p),\\
\dot{p}=-\frac{\partial H}{\partial x}(x,u,p)-\frac{\partial H}{\partial u}(x,u,p)p,\qquad (x,p,u)\in T^*M\times\mathbf{R},\\
\dot{u}=\frac{\partial H}{\partial p}(x,u,p)\cdot p-H(x,u,p).
\end{array}
\right.
\end{align}
Equations (\ref{che}) are the equations of motion for the system from contact Hamiltonian dynamics, which is a natural extension of symplectic Hamiltonian dynamics
\cite{Arn}.  \emph{2)} to introduce a solution semigroup for the evolutionary  first-order partial differential equation
\begin{align}\label{ehj}
w_t+H(x,w,w_x)=0,\quad (x,t)\in M\times[0,+\infty)
\end{align}
for which the characteristic equations are (\ref{che}), provide a representation formula for the semigroup by using the implicit action function and show the existence and uniqueness of the viscosity solution to equation (\ref{ehj}) via the semigroup. \emph{3)} to find pairs $(u,c)$ such that the following stationary first-order partial differential equation
\begin{align}\label{shj}
H(x,u,u_x)=c, \quad x\in M
\end{align}
admits viscosity solutions.

We always assume the contact Hamiltonian $H(x,u,p)$ satisfies the following conditions:
\begin{itemize}
	\item [\textbf{(H1)}] Positive Definiteness: For every $(x,p,u)\in T^*M\times\mathbf{R}$, the second partial derivative $\partial^2 H/\partial p^2 (x,u,p)$ is positive definite as a quadratic form;
	\item [\textbf{(H2)}] Superlinearity: For every $(x,u)\in M\times\mathbf{R}$, $H(x,u,p)$ is  superlinear in $p$;
	\item [\textbf{(H3)}] Lipschitz Continuity: $H(x,u,p)$ is uniformly Lipschitz in $u$, i.e., there exists $\lambda>0$ such that $|\frac{\partial H}{\partial u}(x,u,p)|\leq \lambda$ for any $(x,p,u)\in T^*M\times\mathbf{R}$.
\end{itemize}

In \cite{WWY}, we introduced an implicit variational principle for contact Hamilton's equations (\ref{che}), which is stated as follows.
\begin{proposition}[Implicit Variational Principle]\label{VP}
	For any given $x_0\in M$ and $u_0\in\mathbf{R}$, there exists a continuous function $h_{x_0,u_0}(x,t)$ defined on $M\times (0,+\infty)$ satisfying
	\begin{equation}\label{iaf}
	h_{x_0,u_0}(x,t)=u_0+\inf_{\substack{\gamma(t)=x \\  \gamma(0)=x_0} }\int_0^tL(\gamma(\tau),h_{x_0,u_0}(\gamma(\tau),\tau),\dot{\gamma}(\tau))d\tau,
	\end{equation}
	where the infimum is taken among the Lipschitz continuous curves $\gamma:[0,t]\rightarrow M$ and can be achieved. Let $\gamma$ be a Lipschitz continuous curve achieving the infimum in (\ref{iaf}) and
	\[
	x(s):=\gamma(s),\quad u(s):=h_{x_0,u_0}(x(s),s),\quad p(s):=\frac{\partial L}{\partial \dot{x}}(x(s),u(s),\dot{x}(s)).
	\]
	Then $(x(s),u(s),p(s))$ satisfies equations (\ref{che}) with $x(0)=x_0$, $x(t)=x$ and $\lim_{s\to0^+}u(s)=u_0$.
\end{proposition}
Here, $L$ denotes the contact Lagrangian associated with $H$, see Section 2 for the definition.
The function $h_{x_0,u_0}(x,t)$ is called the implicit action function and the curves achieving the infimum in (\ref{iaf}) are called the minimizers of $h_{x_0,u_0}(x,t)$.

Before stating our main results of the present paper, we would like to recall the notion of a viscosity solution to equations (\ref{ehj}) and (\ref{shj}), which was introduced by Crandall and Lions in \cite{CL2}.
\begin{definition}[Viscosity solution of equation (\ref{ehj})]\label{visco}
	Let $V$ be an open subset  $V\subset M$.
	\begin{itemize}
		\item [(i)] A function $u:V\times[0,+\infty)\rightarrow \mathbf{R}$ is called a viscosity subsolution of equation (\ref{ehj}), if for every $C^1$ function $\varphi:V\times[0,+\infty)\rightarrow\mathbf{R}$ and every point $(x_0,t_0)\in V\times[0,+\infty)$ such that $u-\varphi$ has a local maximum at $(x_0,t_0)$, we have
		\[
		\varphi_t(x_0,t_0)+H(x_0,u(x_0,t_0),\varphi_x(x_0,t_0))\leq 0;
		\]
		\item [(ii)] A function $u:V\times[0,+\infty)\rightarrow \mathbf{R}$ is called a viscosity supersolution of equation (\ref{ehj}), if for every $C^1$ function $\psi:V\times[0,+\infty)\rightarrow\mathbf{R}$ and every point $(y_0,s_0)\in V\times[0,+\infty)$ such that $u-\psi$ has a local minimum at $(y_0,s_0)$, we have
		\[
		\psi_t(y_0,s_0)+H(y_0,u(y_0,s_0),\psi_x(y_0,s_0))\geq 0;
		\]
		\item [(iii)] A function $u:V\times[0,+\infty)\rightarrow\mathbf{R}$ is called a viscosity solution of equation (\ref{ehj}) if it is both a viscosity subsolution and a viscosity supersolution.
	\end{itemize}
\end{definition}
\begin{definition}[Viscosity solution of equation (\ref{shj})]\label{visco}
	Let $U$ be an open subset  $U\subset M$.
	\begin{itemize}
		\item [(i)] A function $u:U\rightarrow \mathbf{R}$ is called a viscosity subsolution of equation (\ref{shj}), if for every $C^1$ function $\varphi:U\rightarrow\mathbf{R}$ and every point $x_0\in U$ such that $u-\varphi$ has a local maximum at $x_0$, we have
		\[
		H(x_0,u(x_0),\varphi_x(x_0))\leq 0;
		\]
		\item [(ii)] A function $u:U\rightarrow \mathbf{R}$ is called a viscosity supersolution of equation (\ref{shj}), if for every $C^1$ function $\psi:U\rightarrow\mathbf{R}$ and every point $y_0\in U$ such that $u-\psi$ has a local minimum at $y_0$, we have
		\[
		H(y_0,u(y_0),\psi_x(y_0))\geq 0;
		\]
		\item [(iii)] A function $u:U\to\mathbf{R}$ is called a viscosity solution of equation (\ref{shj}) if it is both a viscosity subsolution and a viscosity supersolution.
	\end{itemize}
\end{definition}

The first main result of this paper is stated as follows.
\begin{theorem}\label{thehj}
	There is a semigroup of operators $\{T_t\}_{t\geq 0}: C(M,\mathbf{R})\mapsto C(M,\mathbf{R})$, such that
	for each $\varphi\in C(M,\mathbf{R})$, $T_t\varphi(x)$ is the unique viscosity solution of equation (\ref{ehj}) with initial value condition $w(x,0)=\varphi(x)$. Furthermore, we have

	\begin{align}\label{rep-f}
	T_t\varphi(x)=\inf_{y\in M}h_{y,\varphi(y)}(x,t), \quad \forall (x,t)\in M\times[0,+\infty),
	\end{align}
	where $h$ is the implicit action function introduced in Proposition \ref{VP}.
\end{theorem}
The semigroup obtained in Theorem \ref{thehj} can be regarded as a natural generalization of the Lax-Oleinik semigroup for Hamiltonian systems. It connects viscosity solutions of equation (\ref{ehj}) and the implicit action function. The proof of the representation formula (\ref{rep-f}) for the solution semigroup relies on the implicit variational principle---Proposition \ref{VP}. We think that the representation formula has many potential applications.
Here we use it to prove our second main result of the present paper:
\begin{theorem}\label{thshj}
	There exists a constant $c\in\mathbf{R}$ such that equation (\ref{shj}) admits viscosity solutions.
\end{theorem}

In fact, we show that for each $\varphi\in C(M,\mathbf{R})$, there exists a constant $c\in\mathbf{R}$ such that $\liminf_{t\to+\infty}T^c_t\varphi(x)$ denoted by $\varphi_\infty(x)$ is a viscosity solution of equation (\ref{shj}),
where $\{T^c_t\}_{t\geq 0}$ denotes the semigroup of operators associated with $L+c$ obtained in Theorem \ref{thehj}. We prove it by showing that $\varphi_\infty(x)$ is a fixed point of $\{T^c_t\}_{t\geq 0}$.

We prove Theorem \ref{thshj} by using a variational and dynamical approach.
More precisely, we give the proof of the theorem by carefully analysing the properties of the implicit action function $h_{x_0,u_0}(x,t)$ and the semigroup $\{T^{-}_t\}_{t\geq 0}$ which can be represented by $h_{x_0,u_0}(x,t)$. Thus, the key tool used here is the implicit action function $h_{x_0,u_0}(x,t)$. In \cite{WWY}, several important properties of $h_{x_0,u_0}(x,t)$ were discussed, e.g., monotonicity property,  Markov property. In this paper, we will study more interesting properties of $h_{x_0,u_0}(x,t)$.

Theorem \ref{thshj} concerns the so called ergodic problem or additive eigenvalue problem for $H(x,u,p)$, which plays an essential role in homogenization for Hamilton-Jacobi equations, where it is referred to as the cell problem. The classical result in this direction is due to Lions, Papanicolaou and
Varadhan \cite{LPV}. They obtained the existence of the unique constant $c_0$ for which the Hamilton-Jacobi equation
\begin{align}\label{HJ}
H(x,u_x)=c_0,\quad x\in \mathbf{T}^n
\end{align}
has a continuous viscosity solution. Fathi \cite{Fat97, Fat08} generalized this result to equation (\ref{HJ}) on closed, connected and smooth manifolds and his result is now called the weak KAM theorem. A big difference between Theorem \ref{thshj} for equation (\ref{shj}) and the results mentioned above for equation (\ref{HJ}) is that the constant $c$ in Theorem \ref{thshj} may not be unique, while the constant $c_0$ is unique, called the critical value.

The rest of the paper is organized as follows: Section 2 gives some preliminary results. The purpose of Section 3 is to obtain more properties of the implicit action function. First, we will provide a new monotonicity result and a minimality result for the implicit action function. Then we will prove that the function $(x_0,u_0,x,t)\mapsto h_{x_0,u_0}(x,t)$ is Lipschitz  continuous on $M\times [a,b]\times M\times [\delta,T]$, where $a$, $b\in\mathbf{R}$ with $a<b$ and $0<\delta<T$. At last, the reversibility property of the implicit action function can be obtained, which allows us to define another implicit action function. In Section 4, we will give the proof of Theorem \ref{thehj}. More precisely, we will first introduce the forward and backward solution semigroup for equation (\ref{ehj}). Then by using the implicit action functions, representation formulae for the solution semigroups will be provided. Finally, we will discuss the properties of the solution semigroups and the relationship between the semigroups and the viscosity solutions of equation (\ref{ehj}). Section 5 is devoted to the proof of Theorem \ref{thshj}.

\vskip0.2cm

\textbf{Notations}.

\begin{itemize}
	\item $\mathrm{diam}(M)$ denotes the diameter of $M$.
	\item Denote by $d$ the distance induced by the Riemannian metric $g$ on $M$.
	\item Denote by $\|\cdot\|$ the norms induced by $g$ on both tangent and cotangent spaces of $M$.
	\item $C(M,\mathbf{R})$ stands for  the space of continuous functions on $M$, $\|\cdot\|_0$ denotes the supremum norm on it.
	\item  For $T>0$, $C(M\times[0,T],\mathbf{R})$ stands for  the space of continuous functions on $M\times[0,T]$, $\|\cdot\|_\infty$ denotes the supremum norm on it.
	\item $C^{ac}([a,b],M)$ stands for the space of absolutely continuous curves  $[a,b]\to M$.
	\item For each $t\in\mathbf{R}$, $\{t\}=t$ mod 1 denotes the fractional part of $t$ and $[t]$ denotes the greatest integer not greater than $t$.
	\item Given  $a$, $b$, $\delta$, $T\in\mathbf{R}$ with  $a<b$, $0<\delta<T$, let

	\[
	\Omega_{a,b,\delta,T}=M\times [a,b]\times M\times [\delta,T].
	\]
	
\end{itemize}



\section{Preliminaries and definitions}
We recall and prove some preliminary results in this part.
Propositions \ref{Mono I} and \ref{Markov} are the monotonicity property and Markov property of the implicit action function mentioned in the introduction section.

\vskip0.3cm
\noindent\emph{Contact Lagrangians}.
\vskip0.3cm

We can associate to the contact Hamiltonian a Lagrangian denoted by $L(x,u, \dot{x})$, defined by
\[
L(x,u, \dot{x}):=\sup_{p\in T^*_xM}\{\langle \dot{x},p\rangle-H(x,u,p)\},\quad (x,\dot{x},u)\in TM\times\mathbf{R}.
\]
In view of (H1)-(H3), it is straightforward to check that $L$ admits the following properties:
\begin{itemize}
	\item [\textbf{(L1)}] Positive Definiteness: For every $(x,\dot{x},u)\in TM\times\mathbf{R}$, the second partial derivative $\partial^2 L/\partial {\dot{x}}^2 (x,u,\dot{x})$ is positive definite as a quadratic form;
	\item [\textbf{(L2)}] Superlinearity: For every $(x,u)\in M\times\mathbf{R}$, $L(x,u,\dot{x})$ is superlinear in $\dot{x}$;
	\item [\textbf{(L3)}] Lipschitz Continuity: $L(x,u,\dot{x})$ is uniformly Lipschitz in $u$, i.e., there exists $\lambda>0$ such that $|\frac{\partial L}{\partial u}(x,u,\dot{x})|\leq \lambda$ for any $(x,\dot{x},u)\in TM\times\mathbf{R}$.
\end{itemize}

\noindent\emph{Monotonicity  and Markov properties}.
\begin{proposition}[Monotonicity property I \cite{WWY}]\label{Mono I}
	Given $x_0\in M$ and $u_1$, $u_2\in\mathbf{R}$, if $u_1<u_2$, then we have	
	\[
	h_{x_0,u_1}(x,t)<h_{x_0,u_2}(x,t), \quad \forall (x,t)\in M\times (0,+\infty).
	\]
\end{proposition}

\begin{proposition}[Markov property \cite{WWY}]\label{Markov}
	Given $x_0\in M$ and $u_0\in\mathbf{R}$, we have	
	\[
	h_{x_0,u_0}(x,t+s)=\inf_{y\in M}h_{y,h_{x_0,u_0}(y,t)}(x,s)
	\]
	for all  $s$, $t>0$ and all $x\in M$. Moreover, the infimum is attained at $y$ if and only if there exists a minimizer $\gamma$ of $h_{x_0,u_0}(x,t+s)$ with $\gamma(t)=y$.
\end{proposition}

\noindent\emph{A priori compactness estimate}.

Given  $a$, $b$, $\delta$, $T\in\mathbf{R}$ with  $a<b$, $0<\delta<T$, recall that
\[
\Omega_{a,b,\delta,T}=M\times [a,b]\times M\times [\delta,T].
\]
\begin{lemma}[A priori compactness]\label{lem2.1}
	For any given $a$, $b$, $\delta$, $T\in\mathbf{R}$ with  $a<b$, $0<\delta<T$, there exists a compact set $\mathcal{K}:= \mathcal{K}_{a,b,\delta,T}\subset T^*M\times\mathbf{R}$ such that for any $(x_0,u_0,x,t)\in\Omega_{a,b,\delta,T}$ and any minimizer $\gamma(s)$ of $h_{x_0,u_0}(x,t)$, we have
	\[
	(\gamma(s),u(s),p(s))\subset\mathcal{K}, \quad \forall s\in[0,t],
	\]
	where $u(s)=h_{x_0,u_0}(\gamma(s),s)$, $p(s)=\frac{\partial L}{\partial \dot{x}}(\gamma(s),u(s),\dot{\gamma}(s))$ and $\mathcal{K}$ depends only on $a$, $b$, $\delta$ and $T$.
\end{lemma}
We give the proof of Lemma \ref{lem2.1} in  Appendix.

\noindent\emph{Variational solutions}.

In Section 4 we will show that a variational solution of equation (\ref{ehj}) is a viscosity solution. The definition of the variational solution is as follows.
\begin{definition}\label{nw}
	Let $T>0$. A function $u: M\times [0,T]\rightarrow\mathbf{R}$ is called a variational solution of equation (\ref{ehj}) if
	\begin{itemize}
		\item [(i)] for each continuous and piecewise $C^1$ curve $\gamma:[t_1,t_2]\rightarrow M$ with $0\leq t_1<t_2\leq T$, we have
		\[
		u(\gamma(t_2),t_2)-u(\gamma(t_1),t_1)\leq\int_{t_1}^{t_2}L(\gamma(s),u(\gamma(s),s),\dot{\gamma}(s))ds;
		\]
		\item [(ii)] for each $[t_1,t_2]\subset[0,T]$ and each $x\in M$, there exists a $C^1$ curve $\gamma:[t_1,t_2]\rightarrow M$ with $\gamma(t_2)=x$ such that
		\[
		u(x,t_2)-u(\gamma(t_1),t_1)=\int_{t_1}^{t_2}L(\gamma(s),u(\gamma(s),s),\dot{\gamma}(s))ds.
		\]
	\end{itemize}
\end{definition}

\section{Implicit action functions}
In this part, we discuss some fundamental properties of the implicit action function, which are crucial for the proofs of the main results.

\subsection{Monotonicity and minimality}
\begin{proposition}[Monotonicity property II]\label{Mono II}
	Given $L_1$, $L_2$ satisfying (L1)-(L3), $x_0\in M$ and $u_0\in\mathbf{R}$,
	if $L_1<L_2$, then $h^{L_1}_{x_0,u_0}(x,t)<h^{L_2}_{x_0,u_0}(x,t)$ for all $(x,t)\in M\times  (0,+\infty)$,
	where $h^{L_i}_{x_0,u_0}(x,t)$ denotes the implicit action function associated with $L_i$, $i=1,2$.
\end{proposition}

\begin{proof}
	Assume by contradiction that there exists $(x,t)\in M\times(0,+\infty)$ such that $h^{L_1}_{x_0,u_0}(x,t)\geq h^{L_2}_{x_0,u_0}(x,t)$.
	Let $\gamma_2:[0,t]\to M$ be a minimizer of $h^{L_2}_{x_0,u_0}(x,t)$.

	Let $F(s)=h^{L_2}_{x_0,u_0}(\gamma_2(s),s)-h^{L_1}_{x_0,u_0}(\gamma_2(s),s)$ for $s\in(0,t]$.
	From Lemma 3.1 and Lemma 3.2 in \cite{WWY}, we have
	\[
	\lim_{s\rightarrow 0^+}h^{L_2}_{x_0,u_0}(\gamma_2(s),s)=\lim_{s\rightarrow 0^+}h^{L_1}_{x_0,u_0}(\gamma_2(s),s)=u_0.
	\]
	Let $F(0)=0$. Then, $F(s)$ is a continuous function on $[0,t]$ and  $F(t)\leq 0$.

	Note that there exists $s_0\in (0,t)$ such that $F(s_0)\neq 0$. Otherwise, from the continuity of $F$, for any $s\in [0,t]$, we have $F(s)\equiv 0$, i.e., $h^{L_1}_{x_0,u_0}(\gamma_2(s),s)\equiv h^{L_2}_{x_0,u_0}(\gamma_2(s),s)$, it follows from $L_1<L_2$ that
	\begin{align*}
	h^{L_1}_{x_0,u_0}(x,t)=h^{L_2}_{x_0,u_0}(x,t)
	&=u_0+\int_0^{t}L_2(\gamma_2(\tau),h^{L_2}_{x_0,u_0}(\gamma_2(\tau),\tau),\dot{\gamma}_2(\tau))d\tau\\
	&=u_0+\int_0^{t}L_2(\gamma_2(\tau),h^{L_1}_{x_0,u_0}(\gamma_2(\tau),\tau),\dot{\gamma}_2(\tau))d\tau\\
	&>u_0+\int_0^{t}L_1(\gamma_2(\tau),h^{L_1}_{x_0,u_0}(\gamma_2(\tau),\tau),\dot{\gamma}_2(\tau))d\tau,\\
	\end{align*}
	which contradicts
	\[
	h^{L_1}_{x_0,u_0}(x,t)\leq u_0+\int_0^{t}L_1(\gamma_2(\tau),h^{L_1}_{x_0,u_0}(\gamma_2(\tau),\tau),\dot{\gamma}_2(\tau))d\tau.
	\]
	Hence $F(s_0)\neq 0$.

	If  $F(s_0)>0$, in view of $F(t)\leq 0$, there exists $s_1\in (s_0,t]$ such that $F(s_1)=0$ and $F(s)\geq 0$ for $s\in [s_0,s_1]$.
	Since $\gamma_2$ is a minimizer of $h^{L_2}_{x_0,u_0}(x,t)$, then for each $s\in(s_0,s_1)$ we have
	\[
	h^{L_2}_{x_0,u_0}(\gamma_2(s_1),s_1)=h^{L_2}_{x_0,u_0}(\gamma_2(s),s)+\int_s^{s_1}L_2(\gamma_2(\tau),h^{L_2}_{x_0,u_0}(\gamma_2(\tau),\tau),\dot{\gamma}_2(\tau))d\tau,
	\]
	and
	\begin{align*}
	h^{L_1}_{x_0,u_0}(\gamma_2(s_1),s_1)&\leq h^{L_1}_{x_0,u_0}(\gamma_2(s),s)+\int_s^{s_1}L_1(\gamma_2(\tau),h^{L_1}_{x_0,u_0}(\gamma_2(\tau),\tau),\dot{\gamma}_2(\tau))d\tau\\
	&\leq h^{L_1}_{x_0,u_0}(\gamma_2(s),s)+\int_s^{s_1}L_2(\gamma_2(\tau),h^{L_1}_{x_0,u_0}(\gamma_2(\tau),\tau),\dot{\gamma}_2(\tau))d\tau.
	\end{align*}
	Thus,  by (H3) we get $F(s_1)\geq F(s)-\lambda\int_s^{s_1}F(\tau)d\tau$, which together with $F(s_1)=0$ implies
	\[
	F(s)\leq\lambda\int_s^{s_1}F(\tau)d\tau.
	\]
	Let $G(\sigma):=F(s_1-\sigma)$ for $\sigma\in[0,s_1-s_0]$.
	It follows that $G(0)=0$, $G(\sigma)>0$ for $\sigma\in(0,s_1-s_0)$, and
	\[
	G(s_1-s)\leq\lambda\int_{0}^{s_1-s}G(\sigma)d\sigma, s\in[s_0,s_1].
	\]
	By Gronwall inequality, $F(s)=G(s_1-s)\equiv 0$, $\forall s\in[s_0,s_1]$, which contradicts $F(s_0)>0$.

	It remains to exclude the case $F(s_0)<0$. Let $H(s)=-F(s)=h^{L_1}_{x_0,u_0}(\gamma_2(s),s)-h^{L_2}_{x_0,u_0}(\gamma_2(s),s)$, $s\in(0,s_0]$. Then  $H(s)$ is a continuous function on $[0,s_0]$ and $H(0)=0$, $H(s_0)=-F(s_0)>0$. Then there exists $s_2\in (0,s_0]$ such that $H(s_2)=0$ and $H(s)\geq 0$ for $s\in [s_2,s_0]$.
	Since $\gamma_2$ is a minimizer of $h^{L_2}_{x_0,u_0}(x,t)$, we have
	\[
	h^{L_2}_{x_0,u_0}(\gamma_2(s),s)=h^{L_2}_{x_0,u_0}(\gamma_2(s_2),s_2)+\int_{s_2}^{s}L_2(\gamma_2(\tau),h^{L_2}_{x_0,u_0}(\gamma_2(\tau),\tau),\dot{\gamma}_2(\tau))d\tau,
	\]
	and
	\begin{align*}
	h^{L_1}_{x_0,u_0}(\gamma_2(s),s)&\leq h^{L_1}_{x_0,u_0}(\gamma_2(s_2),s_2)+\int_{s_2}^{s}L_1(\gamma_2(\tau),h^{L_1}_{x_0,u_0}(\gamma_2(\tau),\tau),\dot{\gamma}_2(\tau))d\tau\\
	&\leq h^{L_1}_{x_0,u_0}(\gamma_2(s_2),s_2)+\int_{s_2}^{s}L_2(\gamma_2(\tau),h^{L_1}_{x_0,u_0}(\gamma_2(\tau),\tau),\dot{\gamma}_2(\tau))d\tau,
	\end{align*}
	which implies
	\[
	H(s)\leq\lambda\int_{s_2}^{s}H(\tau)d\tau.
	\]
	By Gronwall inequality again, $H(s)\equiv 0$, $\forall s\in[s_2,s_0]$.
	It contradicts $H(s_0)>0$. The proof is now complete.
\end{proof}

\begin{proposition}[Minimizing property]\label{pr3.2}
	Given $x_0$, $x\in M$, $u_0\in\mathbf{R}$ and  $t>0$, let
	\[
	S_{x_0,u_0}^{x,t}=\big\{\mathrm{solutions}\ (x(s),u(s),p(s))\ \mathrm{of}\ (\ref{che})\  \mathrm{on}\ [0,t]:  x(0)=x_0,\ x(t)=x,\ u(0)=u_0\big\}.
	\]
	Then

	\[
	h_{x_0,u_0}(x,t)=\inf\{u(t): (x(s),u(s),p(s))\in S_{x_0,u_0}^{x,t}\}, \quad \forall (x,t)\in M\times(0,+\infty).
	\]
\end{proposition}

\begin{proof}
	By Proposition \ref{VP}, there exists a solution of (\ref{che}) $(\hat{x}(t),\hat{u}(t),\hat{p}(t))$ such that  $\hat{u}(t)=h_{x_0,u_0}(x,t)$ and
	\begin{equation*}
	h_{x_0,u_0}(x,t)=u_0+\int_0^tL(\hat{x}(\tau),h_{x_0,u_0}(\hat{x}(\tau),\tau),\dot{\hat{x}}(\tau))d\tau,
	\end{equation*}
	where $\dot{\hat{x}}(\tau):=\frac{\partial H}{\partial p}(\hat{x}(\tau),\hat{u}(\tau),\hat{p}(\tau))$.

	In the following, we will show that for each solution $(x(s),u(s),p(s))\in S_{x_0,u_0}^{x,t}$,  $u(t)\geq h_{x_0,u_0}(x,t)$.
	Assume by contradiction that there exists $(\tilde{x}(s),\tilde{u}(s),\tilde{p}(s))\in S_{x_0,u_0}^{x,t}$ such that
	$\tilde{u}(t)<h_{x_0,u_0}(x,t)$.  Obviously, we have
	\begin{align*}
	\tilde{u}(t)=u_0+\int_0^{t}L(\tilde{x}(\tau),\tilde{u}(\tau),\dot{\tilde{x}}(\tau))d\tau
	\end{align*}
	and
	\begin{align*}
	h_{x_0,u_0}(x,t)\leq u_0+\int_0^{t}L(\tilde{x}(\tau),h_{x_0,u_0}(\tilde{x}(\tau),\tau),\dot{\tilde{x}}(\tau))d\tau.
	\end{align*}
	Set $v(\sigma):=h_{x_0,u_0}(\tilde{x}(\sigma),\sigma)$ for $\sigma\in [0,t]$. In particular, we have $v(t)=h_{x_0,u_0}(\tilde{x}(t),t)$.  Let
	$F(\sigma)=v(\sigma)-\tilde{u}(\sigma)$, where $\sigma\in [0,t]$. By definition, we have $\tilde{u}(0)=u_0$. In view of Lemma 3.1 and Lemma 3.2 in \cite{WWY}, $v(0)=0$.
	Thus we have $F(0)=0$. The assumption $\tilde{u}(t)<h_{x_0,u_0}(x,t)$ implies $F(t)>0$. Hence, there exists $\sigma_0\in [0,t)$ such that $F(\sigma_0)=0$ and  $F(\sigma)> 0$ for $\sigma> \sigma_0$. Moreover, for any $\tau\in (\sigma_0,t]$, we have
	\begin{equation*}
	\tilde{u}(\tau)=\tilde{u}(\sigma_0)+\int_{\sigma_0}^{\tau}L(\tilde{x}(\sigma),\tilde{u}(\sigma),\dot{\tilde{x}}(\sigma))d\sigma
	\end{equation*}
	and
	\begin{equation*}
	v(\tau)\leq v(\sigma_0)+\int_{\sigma_0}^{\tau}L(\tilde{x}(\sigma),v(\sigma),\dot{\tilde{x}}(\sigma))d\sigma.
	\end{equation*}
	Since $v(\sigma_0)-\tilde{u}(\sigma_0)=F(\sigma_0)=0$, a direct calculation implies
	\begin{equation*}
	v(\tau)-\tilde{u}(\tau)\leq \int_{\sigma_0}^\tau\lambda(v(\sigma)-\tilde{u}(\sigma))d\sigma.
	\end{equation*}
	Hence, we have
	\begin{equation*}
	F(\tau)\leq \int_{\sigma_0}^\tau\lambda F(\sigma)d\sigma.
	\end{equation*}
	Using Gronwall inequality, we have  $F(t)=0$, which contradicts $F(t)=\tilde{u}(t)-h_{x_0,u_0}(x,t)<0$.
\end{proof}

\subsection{Local Lipschitz continuity}
Given  $a$, $b$, $\delta$, $T\in\mathbf{R}$ with  $a<b$, $0<\delta<T$, recall
\[
\Omega_{a,b,\delta,T}=M\times [a,b]\times M\times [\delta,T].
\]
The main result of this part is as follows.

\begin{proposition}\label{pr3.3}
	The function $(x_0,u_0,x,t)\mapsto h_{x_0,u_0}(x,t)$ is Lipschitz  continuous on $\Omega_{a,b,\delta,T}$.
\end{proposition}
This proposition is an immediate consequence of Lemma \ref{lem3.1} and Lemma \ref{lem3.2} below.

\begin{lemma}\label{lem3.1}
	Given any $(x_0,u_0)\in M\times [a,b]$, there exists a constant $\kappa:=\kappa_{a,b,\delta,T}$ such that the function $(x,t)\mapsto h_{x_0,u_0}(x,t)$ is $\kappa$-Lipschitz continuous on $M\times [\delta,T]$.
\end{lemma}

\begin{proof}
	Let $u(\cdot,\cdot)=h_{x_0,u_0}(\cdot,\cdot)$ on $M\times[\delta,T]$. We need to show (i) for any given $t\in[\delta,T]$, $u(\cdot,t)$ is Lipschitz on $M$ with a Lipschitz constant depending only on $a$, $b$, $\delta$ and $T$. (ii) for any given $x\in M$, $u(x,\cdot)$ is Lipschitz on $[\delta,T]$ with a Lipschitz constant depending only on $a$, $b$, $\delta$ and $T$.

	(i) Fix $t\in(\frac{\delta}{4},2T)$, we first show  that for any $x'\in M$, there is a neighborhood $U_{x'}$ of $x'$ and a constant $K_1>0$ depending only on $a$, $b$, $\delta$ and $T$, such that
	\[
	|u(x,t)-u(y,t)|\leq K_1d(x,y),\quad \forall x,\ y\in U_{x'}.
	\]
	Let $\tau=t-\frac{\delta}{4}$, $U_{x'}=B(x',\frac{\tau}{2})$ and $\Delta t=d(x,y)$. Then $d(x,y)\leq \tau$
	and $\Delta t\leq t-\frac{\delta}{4}$. Let $\gamma:[0,t]\to M$ be a minimizer of $u(x,t)$.
\[
	u(y,t)-u(x,t)=\Big(u(y,t)-u(\gamma(t-\Delta t),t-\Delta t)\Big)+\Big(u(\gamma(t-\Delta t),t-\Delta t)-u(x,t)\Big)=:A+B.
	\]

	Next we estimate $A$ and $B$ respectively. For $A$, let $\alpha:[t-\Delta t,t]\to M$ with $\alpha(t-\Delta t)=\gamma(t-\Delta t)$ and $\alpha(t)=y$ be a geodesic with
	\[
	\|\dot{\alpha}\|=\frac{d(\gamma(t-\Delta t),y)}{\Delta t}\leq\frac{d(\gamma(t-\Delta t),x)+d(x,y)}{\Delta t}=\frac{d(\gamma(t-\Delta t),x)}{\Delta t}+1.
	\]
	Note that $d(\gamma(t-\Delta t),x)\leq\int_{t-\Delta t}^t\|\dot{\gamma}\|ds$, which together with Lemma \ref{lem2.1} implies that $d(\gamma(t-\Delta t),x)\leq J_1\Delta t$, where $J_1$ is a constant depending only on $a$, $b$, $\delta$ and $T$. Therefore, $\|\dot{\alpha}\|\leq J_1+1$ for all $s\in[t-\Delta t,t]$. In view of Lemma \ref{lemA} in Appendix, we have $|u(\alpha(s),s)|\leq J_2$ for all $s\in[t-\Delta t,t]$, where $J_2$ is a constant depending only on $a$, $b$, $\delta$ and $T$. Hence,
	\begin{align}\label{3-1}
	A=u(y,t)-u(\gamma(t-\Delta t),t-\Delta t)\leq \int_{t-\Delta t}^t L(\alpha(s),u(\alpha(s),s),\dot{\alpha}(s))ds\leq J_3d(x,y),
	\end{align}
	for some constant $J_3$ depending only on $a$, $b$, $\delta$ and $T$.

	For $B$, we have
	\begin{align*}
	u(\gamma(t-\Delta t),t-\Delta t)-u(x,t)&=-\int_{t-\Delta t}^t L(\gamma(s),u(\gamma(s),s),\dot{\gamma}(s))ds\\
&\leq \int_{t-\Delta t}^t |L(\gamma(s),u(\gamma(s),s),\dot{\gamma}(s))|ds.
	\end{align*}
	By Lemma \ref{lem2.1}, we have
	\begin{align}\label{3-2}
	B\leq J_4d(x,y)
	\end{align}
	for some constant $J_4$ depending only on $a$, $b$, $\delta$ and $T$.

	Combining (\ref{3-1}) and (\ref{3-2}), we have
	\[
	u(y,t)-u(x,t)\leq (J_3+J_4)d(x,y).
	\]
	By exchanging the roles of $x$ and $y$, we get
	\[
	|u(y,t)-u(x,t)|\leq K_1d(x,y),\quad \forall (x,y)\in U_{x'},
	\]
	where constant $K_1$ depends only on $a$, $b$, $\delta$ and $T$. This means that $u(\cdot,t)$ is locally Lipschitz continuous on $M$ for any given $t\in(\frac{\delta}{4},2T)$. Since $M$ is compact and the existence of geodesic between arbitrary $x$ and $y$, we conclude that for any given $t\in(\frac{\delta}{4},2T)$, $u(\cdot,t)$ is Lipschitz continuous on $M$ with a Lipschitz constant depending  on $a$, $b$, $\delta$ and $T$ only.

	(ii) Fix $x\in M$, for any $t'\in [\delta,T]$, we show that there is a neighborhood $V_{t'}$ of $t'$ and a constant $K_2>0$ depending only on $a$, $b$, $\delta$ and $T$, such that
	\[
	|u(x,t)-u(x,s)|\leq K_2|t-s|,\quad \forall t,\ s\in V_{t'}.
	\]
	Let $\xi=t'-\frac{\delta}{2}$ and $V_{t'}=[t'-\xi,t'+\xi]$. Then $\frac{\delta}{2}\leq \xi\leq T$ and
	$V_{t'}=[\frac{\delta}{2},2t'-\frac{\delta}{2}]\subset[\frac{\delta}{2},2T]$. Now we estimate $u(x,t)-u(x,s)$, $t$, $s\in V_{t'}$.

	If $t>s$, let $\gamma:[0,t]\to M$ be a minimizer of $u(x,t)$. Then
	\[
	u(x,t)=u(\gamma(s),s)+\int_s^tL(\gamma(\tau),u(\gamma(\tau),\tau),\dot{\gamma}(\tau))d\tau.
	\]
	By Lemma \ref{lem2.1} we have $|L(\gamma(\tau),u(\gamma(\tau),\tau),\dot{\gamma}(\tau))|\leq J_5$
	for some constant $J_5>0$ depending only on $a$, $b$, $\delta$ and $T$. Thus, we get
	\begin{align*}
	u(x,t)-u(x,s)&=u(\gamma(s),s)-u(x,s)+\int_s^tL(\gamma(\tau),u(\gamma(\tau),\tau),\dot{\gamma}(\tau))d\tau\\
	&\leq u(\gamma(s),s)-u(x,s)+J_5(t-s).
	\end{align*}
	From (i) and Lemma \ref{lem2.1}, we have
	\[
	|u(\gamma(s),s)-u(x,s)|\leq K_1d(\gamma(s),x)\leq K_1\int_s^t\|\dot{\gamma}\|d\tau\leq J_6(t-s)
	\]
	for some constant $J_6>0$ depending only on $a$, $b$, $\delta$ and $T$. Therefore, we have
	\[
	u(x,t)-u(x,s)\leq K_2(t-s)
	\]
	for some constant $K_2>0$ depending only on $a$, $b$, $\delta$ and $T$.

	If $t<s$, we can obtain
	\[
	u(x,t)-u(x,s)\leq K_2(s-t)
	\]
	in a similar manner.

	Therefore, we have
	\[
	|u(x,t)-u(x,s)|\leq K_2|t-s|.
	\]
	This implies that $u(x,\cdot)$ is locally Lipschitz continuous on $[\delta,T]$ for any given $x\in M$ with
	a Lipschitz constant depending only on $a$, $b$, $\delta$ and $T$. From the compactness of $[\delta,T]$, we conclude that $u(x,\cdot)$ is Lipschitz continuous on $[\delta,T]$ for any given $x\in M$ with
	a Lipschitz constant depending only on $a$, $b$, $\delta$ and $T$.
	
\end{proof}

\begin{lemma}\label{lem3.2}
	Given any $(x,t)\in M\times [\delta,T]$, there exists a constant $\iota:=\iota_{a,b,\delta,T}$ such that the function $(x_0,u_0)\mapsto h_{x_0,u_0}(x,t)$ is $\iota$-Lipschitz continuous on $M\times[a,b]$.
\end{lemma}

\begin{proof}
	By Lemma \ref{lem2.1}, there exists a compact set $\mathcal{K}:= \mathcal{K}_{a,b,\delta,T}\subset T^*M\times\mathbf{R}$ such that for any $(x_0,u_0,x,t)\in\Omega_{a,b,\delta,T}$ and any minimizer $\gamma(s)$ of $h_{x_0,u_0}(x,t)$, we have
	\[
	(x(s),u(s),p(s))\subset\mathcal{K}, \quad \forall s\in[0,t],
	\]
	where $x(s)=\gamma(s)$, $u(s)=h_{x_0,u_0}(x(s),s)$ and $p(s)=\frac{\partial L}{\partial \dot{x}}(x(s),u(s),\dot{x}(s))$. Let $V$ be a neighborhood of $\mathcal{K}$.

	Given any $(x,t)\in M\times [\delta,T]$, a point $(x_0,u_0,p_0)$ is called a minimizing point of $(x,t)$ if
	a solution $(x(s),u(s),p(s))$ of equations (\ref{che}) with $x(0)=x_0$, $u(0)=u_0$, $x(t)=x$ and $\frac{\partial L}{\partial\dot{x}}(x_0,u_0,\dot{x}(0))=p_0$ exists on $[0,t]$ and $x(s)$ is a minimizer of $h_{x_0,u_0}(x,t)$. Proposition \ref{VP} guarantees the existence of minimizing points of $(x,t)$.
	Let $\mathcal{G}$ denote the set of minimizing points of $(x,t)$. It is clear that $\mathcal{G}$ is a compact subset of $\mathcal{K}$. From the theory of ordinary differential equations, there is a constant $\Delta>0$ such that for each $(x_0,u_0,p_0)\in \mathcal{G}$, if $d((x,u,p),(x_0,u_0,p_0))<\Delta$, then the solution $(x(s),u(s),p(s))$ of equations (\ref{che}) with $(x,u,p)$ as initial value condition exists on $[0,t]$ and
	$(x(s),u(s),p(s))\subset V$.

	Given any $(x_0,u_0)\in M\times[a,b]$, let $U=\{(x,u)\in M\times[a,b]\ |\ d\big((x,u),(x_0,u_0)\big)<\frac{\Delta}{2}\}$.
	For each $(x_1,u_1)$, $(x_2,u_2)\in U$, it suffices to show that
	\begin{align}\label{3-3}
	|h_{x_1,u_1}(x,t)-h_{x_2,u_2}(x,t)|\leq\iota(d(x_1,x_2)+|u_1-u_2|)
	\end{align}
	for some constant $\iota>0$ which depends only on $a$, $b$, $\delta$ and $T$.

	By Proposition \ref{VP}, there is a minimizer $x_1(s)$ of $h_{x_1,u_1}(x,t)$. Let $p_1=\frac{\partial L}{\partial \dot{x}}(x_1,u_1,\dot{x}_1(0))$ and $(x_1(s),u_1(s),p_1(s))$ denote the solution of equations (\ref{che}) with $(x_1,u_1,p_1)$ as initial conditions. Since $(x_1,u_1,p_1)\in\mathcal{G}$ and $d\big((x_1,u_1,p_1),(x_2,u_2,p_1)\big)<\Delta$, then the solution $(x_2(s),u_2(s),p_2(s))$ of equations (\ref{che}) with $(x_2,u_2,p_1)$ as initial conditions exists on $[0,t]$.

	By the differentiability of the solutions of equations (\ref{che}) with respect to initial values, there is a constant $C>0$ depending only on $V$ such that
	\begin{align}\label{3-4}
	\begin{split}
	d\big(x_1(t),x_2(t)\big)&\leq C\big(d(x_1,x_2)+|u_1-u_2|\big),\\
	|u_1(t)-u_2(t)|&\leq C\big(d(x_1,x_2)+|u_1-u_2|\big).
	\end{split}
	\end{align}
	From Lemma \ref{lem3.1} and (\ref{3-4}), we have
	\begin{align}\label{3-5}
	|h_{x_2,u_2}(x,t)-h_{x_2,u_2}(x_2(t),t)|\leq \kappa d\big(x,x_2(t)\big)\leq\kappa C \big(d(x_1,x_2)+|u_1-u_2|\big).
	\end{align}
	In view of Proposition \ref{pr3.2} and (\ref{3-4}), we get
	\begin{align}\label{3-6}
	h_{x_2,u_2}(x_2(t),t)\leq u_2(t)\leq u_1(t)+C\big(d(x_1,x_2)+|u_1-u_2|\big).
	\end{align}
	Note that $u_1(t)=h_{x_1,u_1}(x,t)$. Thus, combining (\ref{3-5}) and (\ref{3-6}), we have
	\[
	h_{x_2,u_2}(x,t)\leq h_{x_1,u_1}(x,t)+C(\kappa+1)\big(d(x_1,x_2)+|u_1-u_2|\big).
	\]

	By exchanging the roles of $(x_1,u_1)$ and $(x_2,u_2)$, one can show (\ref{3-3}) which completes the proof.
\end{proof}

\noindent\emph{Proof of Proposition \ref{pr3.3}}.
For each $(x_1,u_1,y_1,t_1)$, $(x_2,u_2,y_2,t_2)\in \Omega_{a,b,\delta,T}$. It follows from Lemma \ref{lem3.1} and Lemma \ref{lem3.2} that
\begin{align*}
|h_{x_1,u_1}(y_1,t_1)-h_{x_2,u_2}(y_2,t_2)|&\leq|h_{x_1,u_1}(y_1,t_1)-h_{x_2,u_2}(y_1,t_1)|+|h_{x_2,u_2}(y_1,t_1)-h_{x_2,u_2}(y_2,t_2)|\\
&\leq \iota d\big((x_1,x_2)+|u_1-u_2|\big)+\kappa \big(d(y_1,y_2)+|t_1-t_2|\big)\\
&\leq l\Big(d(x_1,x_2)+|u_1-u_2|+d(y_1,y_2)+|t_1-t_2|\Big),
\end{align*}
where $l:=\max\{\iota,\kappa\}$. This completes the proof of Proposition \ref{pr3.3}. \hfill$\Box$

\subsection{Reversibility}
We will introduce another implicit action function based on the following property of $h_{x_0,u_0}(x,t)$.
\begin{proposition}[Reversibility property]\label{pr3.4}
	Given $x_0$, $x\in M$, and $t\in(0,+\infty)$, for each $u\in \mathbf{R}$, there exists a unique $u_0\in \mathbf{R}$ such that
	\[
	h_{x_0,u_0}(x,t)=u.
	\]
\end{proposition}

\begin{proof}
	In view of Proposition \ref{Mono I}, we only need to prove the existence of $u_0$.
	By the Lipschitz continuity of $h_{x_0,u_0}(x,t)$ with respect to $u_0$ given by Proposition \ref{pr3.3} , it suffices to show that for each $A>0$ large enough, one can find $u_1$, $u_2\in \mathbf{R}$ such that (i) $h_{x_0,u_1}(x,t)\geq A$ and (ii) $h_{x_0,u_2}(x,t)\leq -A$.

	For (i), let $\gamma:[0,t]\rightarrow M$ be a minimizer of $h_{x_0,u_1}(x,t)$, where $u_1$ is a constant to be determined. By Proposition \ref{VP}, $u(s):=h_{x_0,u_1}(\gamma(s),s)$ is of class $C^1$ for $s\in (0,t]$ and $\lim_{s\rightarrow 0^+}u(s)=u_1$. Let $d_1=\inf_{(x,\dot{x})\in TM} L(x,0,\dot{x})$. From (L2) $d_1$ is well defined. Since $L$ satisfies $(L3)$ and $(\gamma(s),u(s),p(s))$ satisfies equations (\ref{che}), where $p(s)=\frac{\partial L}{\partial \dot{x}}(\gamma(s),u(s),\dot{\gamma}(s))$, then we have
	\[
	\dot{u}(s)=L(\gamma(s),u(s),\dot{\gamma}(s))\geq d_1-\lambda|u(s)|.
	\]
	Consider the Cauchy problem
	\begin{equation*}
	\begin{cases}
	\dot{v}(s)= d_1-\lambda v(s),\quad  s\in(0,t],\\
	v(0)=u_1.
	\end{cases}
	\end{equation*}
	We have
	\[
	v(t)=u_1e^{-\lambda t}+\frac{d_1}{\lambda}(1-e^{-\lambda t}).
	\]
	Requiring $v(t)\geq A$, it yields $u_1\geq Ae^{\lambda t}-\frac{d_1}{\lambda}(e^{\lambda t}-1)$. In order to make $v(s)\geq 0$ for $s\in [0,t]$, we take $u_1\geq \frac{|d_1|}{\lambda}\max\{e^{\lambda t},1\}$. Moreover, let
	\[
	u_1=\max\{Ae^{\lambda t}-\frac{d_1}{\lambda}(e^{\lambda t}-1),\frac{|d_1|}{\lambda}\max\big\{e^{\lambda t},1\}\big\}.
	\]
	The comparison theorem of ordinary differential equations implies $u(t)\geq A$, i.e., $h_{x_0,u_1}(x,t)\geq A$.

	For (ii), let $\bar{\gamma}:[0,t]\rightarrow M$ be a geodesic between $x_0$ and $x$
	with $\|\dot{\bar{\gamma}}\|=\frac{d(x_0,x)}{t}$.
	Let $w(s)=h_{x_0,u_2}(\bar{\gamma}(s),s)$ for $(0,t]$. In particular, $w(t)=h_{x_0,u_2}(x,t)$, where $u_2$ is constant to be determined. Let $d_2=\max_{\|\dot{x}\|\leq \frac{\mathrm{diam}(M)}{t}}L(x,0,\dot{x})$. By the definition of implicit action functions and (L3),
	for each $s_1$, $s_2\in (0,t]$ with $s_1<s_2$, we get
	\[
	w(s_2)\leq w(s_1)+\int_{s_1}^{s_2}(d_2+\lambda|w(s)|)ds.
	\]
	Note that $w(s)$ is Lipschitz continuous, then for almost all $s\in (0,t]$, we have
	\[
	\dot{w}(s)\leq d_2+\lambda |w(s)|.
	\]
	Consider
	\begin{equation*}
	\begin{cases}
	\dot{v}(s)= d_2-\lambda v(s),\quad s\in(0,t],\\
	v(0)=u_2.
	\end{cases}
	\end{equation*}
	We have
	\[
	v(t)=u_2e^{-\lambda t}+\frac{d_2}{\lambda}(1-e^{-\lambda t}).
	\]
	Requiring $v(t)\leq-A$, it yields $u_2\leq -Ae^{\lambda t}-\frac{d_2}{\lambda}(e^{\lambda t}-1)$. In order to make $v(s)\leq 0$ for $s\in [0,t]$, we take $u_2\leq - \frac{|d_2|}{\lambda}\max\{e^{\lambda t},1\}$. Moreover, let
	\[
	u_2=\min\{-Ae^{\lambda t}-\frac{d_2}{\lambda}(e^{\lambda t}-1),-\frac{|d_2|}{\lambda}\max\{e^{\lambda t},1\}\}.
	\]
	Using the comparison theorem of ordinary differential equations again, we have $w(t)\leq -A$. That is $h_{x_0,u_2}(x,t)\leq -A$.
\end{proof}

We can associate to the implicit action function $h_{x_0,u_0}(x,t)$ a new implicit action function $h^{x_0,u_0}(x,t)$ well defined by
\begin{align}\label{3-8}
h^{x_0,u_0}(x,t)=u_0-\inf_{\substack{\gamma(t)=x_0 \\  \gamma(0)=x } }\int_0^tL(\gamma(\tau),h^{x_0,u_0}(\gamma(\tau),t-\tau),\dot{\gamma}(\tau))d\tau,
\end{align}
where the infimum is taken among the Lipschitz continuous curves $\gamma:[0,t]\rightarrow M$. We call $h_{x_0,u_0}(x,t)$ and $h^{x_0,u_0}(x,t)$ the \emph{forward}  and \emph{backward}  implicit action functions respectively. By arguments similar to the ones we made for $h_{x_0,u_0}(x,t)$ in \cite{WWY} and in the present work, we have

\begin{theorem}
	For any given $x_0\in M$ and $u_0\in\mathbf{R}$, there exists a  continuous function $h^{x_0,u_0}(x,t)$ defined on $M\times(0,+\infty)$
	satisfying (\ref{3-8}). Moreover, the infimum in (\ref{3-8}) can be achieved. If  $\gamma$ is a Lipschitz curve achieving the infimum, let $x(s):=\gamma(s)$, $u(s):=h^{x_0,u_0}(x(s),s)$, $p(s):=\frac{\partial L}{\partial \dot{x}}(x(s),u(s),\dot{x}(s))$.
	Then $(x(s),u(s),p(s))$ satisfies equations (\ref{che}) with $x(0)=x$, $x(t)=x_0$ and $\lim_{s\rightarrow t^-}u(s)=u_0$. Furthermore, $h^{x_0,u_0}(x,t)$ has the following properties.
	\begin{itemize}
		\item
		Given $x_0\in M$ and $u_1$, $u_2\in\mathbf{R}$, $L_1$, $L_2$ satisfying (L1)-(L3), we have
		\begin{itemize}
			\item [(i)]
			if $u_1<u_2$, then $h^{x_0,u_1}(x,t)<h^{x_0,u_2}(x,t)$, for all $(x,t)\in M\times (0,+\infty)$;
			\item [(ii)]if $L_1<L_2$, then $h_{L_1}^{x_0,u_0}(x,t)<h_{L_2}^{x_0,u_0}(x,t)$, for all $(x,t)\in M\times  (0,+\infty)$, where $h_{L_i}^{x_0,u_0}(x,t)$ denotes the backward implicit action function associated with $L_i$, $i=1,2$.
		\end{itemize}
		\item
		Given $x_0$, $x\in M$, $u_0\in\mathbb{R}$ and $t>0$, let
$S_{x,t}^{x_0,u_0}$ be the set of the solutions $(x(s),u(s),p(s))$ of  (\ref{che}) on $[0,t]$ with $x(0)=x$, $x(t)=x_0$, $u(t)=u_0$.
Then
\[
h^{x_0,u_0}(x,t)=\sup\{u(0): (x(s),u(s),p(s))\in S_{x,t}^{x_0,u_0}\}, \quad \forall (x,t)\in M\times(0,+\infty).
\]
		\item
		The function $(x_0,u_0,x,t)\mapsto h^{x_0,u_0}(x,t)$ is Lipschitz  continuous on $\Omega_{a,b,\delta,T}$.
		\item
		Given $x_0\in M$, $u_0\in\mathbf{R}$, we have
		\[
		h^{x_0,u_0}(x,t+s)=\sup_{y\in M}h^{y,h^{x_0,u_0}(y,t)}(x,s)
		\]
		for all  $s$, $t>0$ and all $x\in M$. Moreover, the supremum is attained at $y$ if and only if there exists a minimizer $\gamma$ of $h^{x_0,u_0}(x,t+s)$, such that $\gamma(t)=y$.
		\item
		Given $x_0$, $x\in M$, and $t\in(0,+\infty)$, for each $u\in \mathbf{R}$, there exists a unique $u_0\in \mathbf{R}$ such that
		\[
		h^{x_0,u_0}(x,t)=u.
		\]
	\end{itemize}
\end{theorem}

Moreover, we obtain the relation between $h_{x_0,u_0}(x,t)$ and $h^{x_0,u_0}(x,t)$ in the following.
\begin{proposition}\label{relation}
	\[h_{x_0,u_0}(x,t)=u\Leftrightarrow h^{x,u}(x_0,t)=u_0.\]
\end{proposition}

\begin{proof}
	We verify
	\[h_{x_0,u_0}(x,t)=u\Rightarrow h^{x,u}(x_0,t)=u_0.\]
	The converse implication is similar. Let
	\[
	S_{x_0,t}^{x,u}:=\big\{\mathrm{solutions}\ (x(s),u(s),p(s))\ \mathrm{of}\ (\ref{che})\  \mathrm{on}\ [0,t]:  x(0)=x_0,\ x(t)=x,\ u(t)=u\big\}.
	\]Then
	\[
	h^{x,u}(x_0,t)=\sup\{u(0): (x(s),u(s),p(s))\in S_{x_0,t}^{x,u}\}, \quad \forall (x_0,t)\in M\times(0,+\infty),
	\]
	which shows $h^{x,u}(x_0,t)\geq u_0$. Let $h^{x,u}(x_0,t)=u_1>u_0$. We argue by contradiction. Let $\gamma:[0,t]\rightarrow M$ be a minimizer of $h^{x,u}(x_0,t)$ with $\gamma(0)=x_0$ and $\gamma(t)=x$. We denote
	\[F(s)=h^{x,u}(\gamma(s),t-s)-h_{x_0,u_0}(\gamma(s),s).\]
	It follows that $F(0)=u_1-u_0>0$ and $F(t)=u-u=0$. Hence, there exists $s_0\in (0,t]$ such that $F(s_0)=0$ and $F(s)> 0$ for $s\in [0,s_0)$. Based on the minimality of $\gamma$, we have
	\[h^{x,u}(\gamma(s),t-s)=h^{x,u}(\gamma(s_0),t-s_0)-\int_s^{s_0}L(\gamma(\tau),h^{x,u}(\gamma(\tau),t-\tau),\dot{\gamma}(\tau))d\tau,\]
	\[h_{x_0,u_0}(\gamma(s_0),s_0)\leq h_{x_0,u_0}(\gamma(s),s)+\int_s^{s_0}L(\gamma(\tau),h_{x_0,u_0}(\gamma(\tau),\tau),\dot{\gamma}(\tau))d\tau.\]
	It yields
	\[F(s)\leq \lambda\int_s^{s_0}F(\tau)d\tau.\]
	By Gronwall inequality, we have $F(s)\leq 0$ for $s\in [0,s_0]$, which contradicts $F(0)=0$. Then $h^{x,u}(x_0,t)=u_0$.
\end{proof}



\section{Application I: Solution semigroups for $w_t+H(x,w,w_x)=0$}
In this part, we will consider the following Cauchy problem
\begin{equation}\label{Cau}
\begin{cases}
w_t+H(x,w,w_x)=0,\quad (x,t)\in M\times(0,+\infty),\\
w(x,0)=\varphi(x), \quad x\in M.
\end{cases}
\end{equation}
Our goal is to prove Theorem \ref{thehj} which is an immediate consequence of Propositions \ref{pr4.1}, \ref{pr4.3} and \ref{pr4.5}.

\subsection{Solution semigroups}
Given $\varphi\in C(M,\mathbf{R})$ and  $T>0$, we define an operator $A_\varphi:C(M\times[0,T],\mathbf{R})\to C(M\times[0,T],\mathbf{R})$ by
\begin{equation*}
\forall u\in C(M\times[0,T],\mathbf{R}), \quad
A_\varphi[u](x,t) = \inf_\gamma\{\varphi\big(\gamma(0)\big) + \int_0^t
L\big(\gamma(s),u\big(\gamma(s),s\big), \dot\gamma(s)\big) ds\},
\end{equation*}
where the infimum is taken among the Lipschitz continuous curves $\gamma:[0,t]\to M$ with $\gamma(t)=x$.
By Tonelli Theorem (see for instance \cite{BGH}) the above infimum can be achieved.

\begin{lemma}\label{lem4.1}
	For any given $\varphi\in C(M,\mathbf{R})$ and  $T>0$,  $A_\varphi$ admits a unique fixed point.
\end{lemma}

\begin{proof}
	For any $v\in C(M\times[0,T],\mathbf{R})$ and any $(x,t)\in M\times[0,T]$, by Tonelli Theorem, there exists $\gamma_1:[0,t]\to M$ such that $\gamma_1(t)=x$ and
	\begin{equation*}
	A_\varphi[v](x,t) = \varphi\big(\gamma_1(0)\big) + \int_0^t L\big(\gamma_1(s),v\big(\gamma_1(s),s\big), \dot\gamma_1(s)\big) ds.
	\end{equation*}
	For any $u\in C(M\times[0,T],\mathbf{R})$, from (L3)  we have
	\begin{align*}
	\big(A_\varphi[u] - A_\varphi[v]\big)(x,t)& \leq\int_0^t (L\big(\gamma_1(s),u\big(\gamma_1(s),s\big),\dot\gamma_1(s)\big) - L\big(\gamma_1(s),v\big(\gamma_1(s),s\big), \dot\gamma_1(s)\big) )ds\\
	&\leq \lambda\|u-v\|_{\infty} t.
	\end{align*}

	By exchanging the position of $u$ and $v$, we obtain
	\begin{equation}\label{4-1}
	|\big(A_\varphi[u] - A_\varphi[v]\big)(x,t)| \leq \lambda \ \|u-v\|_{\infty}t.
	\end{equation}
	Let $\gamma_2:[0,t]\rightarrow M$ be the curve such that
	\[
	A_\varphi^2[v](x,t) = \varphi\big(\gamma_2(0)\big) + \int_0^t L\big(\gamma_2(s),A_\varphi [v]\big(\gamma_2(s),s\big), \dot\gamma_2(s)\big) ds.
	\]
	It follows from (\ref{4-1}) that  for $s\in [0,t]$, we have
	\begin{equation*}
	|\big(A_\varphi[u] - A_\varphi[v]\big)(\gamma_2(s),s)| \leq \lambda \ \|u-v\|_{\infty}s.
	\end{equation*}
	Moreover, we have the following estimates
	\begin{align*}
	|\big(A_\varphi^2[u] - A_\varphi^2[v]\big)(x,t)|\leq & \int_0^t \lambda | A_\varphi[u]\big(\gamma_2(s),s\big) -
	A_\varphi[v]\big(\gamma_2(s),s\big)| ds\\
	\leq & \int_0^t  s \lambda^2 \|u-v\|_{\infty} ds \leq \frac{(t\lambda)^2}{2}
	\|u-v\|_{\infty}.
	\end{align*}
	Moreover, continuing the above procedure, we obtain
	\[
	|\big(A_\varphi^n[u] - A_\varphi^n[v]\big)(x,t)| \leq\frac{(t\lambda)^n}{n!} \|u-v\|_{\infty},
	\]
	which implies
	\[
	\|A_\varphi^n[u] - A_\varphi^n[v]\|_{\infty}\leq\frac{(T\lambda)^n}{n!} \|u-v\|_{\infty}.
	\]
	Therefore, there exists  $N\in \mathbf{N}$ large enough such that $A_\varphi^{N}$ is a
	contraction. Thus, there exists a $u(x,t)\in C(M\times [0,T], \mathbf{R})$ such that
	\begin{equation*}
	A_\varphi^N[u]=u.
	\end{equation*}
	Since $A_\varphi[u] = A_\varphi\circ A_\varphi^N[u] = A_\varphi^N\circ A_\varphi[u]$, we have $A_\varphi[u]$ is also a fixed point of $A_\varphi^N$.
	By the uniqueness of fixed point of $A_\varphi^N$, we have
	\[
	A_\varphi[u] = u.
	\]
	This completes the proof of Lemma \ref{lem4.1}.
\end{proof}

We are now in a position to introduce the solution semigroup for (\ref{Cau}).
We define a family of nonlinear operators $\{T^-_t\}_{t\geq 0}$ from $C(M,\mathbf{R})$ to itself as follows.

\begin{definition}[Backward semigroup]
	For each $\varphi\in C(M,\mathbf{R})$, let
	\[
	T^-_t\varphi(x)=u(x,t), \quad\forall (x,t)\in M\times[0,+\infty),
	\]
	where $u(x,t)$ is the unique fixed point obtained in Lemma \ref{lem4.1}.
\end{definition}

By definition, we have
\begin{equation*}\label{fixubac}
T^-_t\varphi(x)=\inf_{\gamma}\{\varphi(\gamma(0))+\int_0^tL(\gamma(\tau),T^-_\tau\varphi(\gamma(\tau)),\dot{\gamma}(\tau))d\tau\},
\end{equation*}
where the infimum is taken among the Lipschitz continuous curves $\gamma:[0,t]\to M$ with $\gamma(t)=x$ and can be achieved. We call the curves achieving the infimum minimizers of $T^-_t\varphi(x)$.
We will show $\{T^-_t\}_{t\geq 0}$ is a semigroup of operators later and call it the \emph{backward solution semigroup} for (\ref{Cau}).

Similarly, we can define another semigroup of operators $\{T^+_t\}_{t\geq 0}$, called the \emph{forward semigroup}, by
\begin{equation*}\label{fixufor}
T^+_t\varphi(x)=\sup_{\gamma}\{\varphi(\gamma(t))-\int_0^tL(\gamma(\tau),T^+_{t-\tau}\varphi(\gamma(\tau)),\dot{\gamma}(\tau))d\tau\},
\end{equation*}
where the infimum is taken among the Lipschitz continuous curves $\gamma:[0,t]\to M$ with $\gamma(0)=x$.

\subsection{Solution semigroups and implicit action functions}
We study the relationship between solution semigroups and implicit action functions here. First, we give a representation formula for the solution semigroup $\{T^-_t\}_{t\geq 0}$ as follows.
\begin{proposition}[Representation formula]\label{pr4.1}
	For each  $\varphi\in C(M,\mathbf{R})$, we have
	\begin{equation*}
	T^{-}_t\varphi(x)=\inf_{y\in M}h_{y,\varphi(y)}(x,t),\quad  \forall (x,t)\in M\times(0,+\infty).
	\end{equation*}
\end{proposition}

\begin{proof}
	Let $u(x,t):=T^-_t\varphi(x)$. It suffices to show the following inequality
	\[
	u(x,t)\geq\inf_{y\in M}h_{y,\varphi(y)}(x,t),
	\]
	since the proof of $u(x,t)\leq\inf_{y\in M}h_{y,\varphi(y)}(x,t)$ follows in a similar manner.

	Let $\gamma_1:[0,t]\rightarrow M$ be a minimizer of $u(x,t)$. Set $\bar{y}=\gamma_1(0)$.
	It is  sufficient to show
	\[
	u(x,t)\geq h_{\bar{y},\varphi(\bar{y})}(x,t).
	\]
	Assume by contradiction that $u(x,t)< h_{\bar{y},\varphi(\bar{y})}(x,t)$.
	Since $\gamma_1$ is a minimizer of $u(x,t)$ and in view of the definition of $h_{y,\varphi(y)}(x,t)$, we have
	\begin{align*}
	u(x,t)=\varphi(\bar{y})+\int_0^{t}L(\gamma_1(\tau),u(\gamma_1(\tau),\tau),\dot{\gamma}_1(\tau))d\tau
	\end{align*}
	and
	\begin{align*}
	h_{\bar{y},\varphi(\bar{y})}(x,t)\leq\varphi(\bar{y})+
	\int_0^{t}L(\gamma_1(\tau),h_{\bar{y},\varphi(\bar{y})}(\gamma_1(\tau),\tau),\dot{\gamma}_1(\tau))d\tau.
	\end{align*}
	Set $\bar{u}(\sigma)=u(\gamma_1(\sigma),\sigma)$ and $\bar{h}(\sigma)=h_{\bar{y},\varphi(\bar{y})}(\gamma_1(\sigma),\sigma)$ for $\sigma\in[0,t]$. In particular, we have $\bar{u}(t)=u(x,t)$ and $\bar{h}(t)=h_{\bar{y},\varphi(\bar{y})}(x,t)$.  Let
	\[
	F(\sigma):=\bar{h}(\sigma)-\bar{u}(\sigma), \quad \sigma\in[0,t].
	\]
	Note that $\bar{u}(0)=\varphi(\bar{y})$. From Lemma 3.1 and Lemma 3.2 in \cite{WWY}, we get $\bar{h}(0)=\varphi(\bar{y})$. Then $F(0)=0$ and  $F(t)>0$. Hence, there exists $\sigma_0\in [0,t)$ such that $F(\sigma_0)=0$ and  $F(\sigma)> 0$ for $\sigma\in (\sigma_0,t]$. Moreover, for any $\tau\in (\sigma_0,t]$, we have
	\[
	\bar{u}(\tau)=\bar{u}(\sigma_0)+\int_{\sigma_0}^{\tau}L(\gamma_1(\sigma),\bar{u}(\sigma),\dot{\gamma}_1(\sigma))d\sigma,
	\]
	and
	\[
	\bar{h}(\tau)\leq \bar{h}(\sigma_0)+\int_{\sigma_0}^{\tau}L(\gamma_1(\sigma),\bar{h}(\sigma),\dot{\gamma}_1(\sigma))d\sigma.
	\]
	Since $\bar{h}(\sigma_0)-\bar{u}(\sigma_0)=F(\sigma_0)=0$, a direct calculation implies
	\[
	\bar{h}(\tau)-\bar{u}(\tau)\leq \int_{\sigma_0}^\tau\lambda(\bar{h}(\sigma)-\bar{u}(\sigma))d\sigma.
	\]
	Hence, we have
	\[
	F(\tau)\leq \int_{\sigma_0}^\tau\lambda F(\sigma)d\sigma,
	\]
	which together with Gronwall inequality implies $F(t)\leq 0$. It contradicts $F(t)>0$. Hence, we have
	\[
	u(x,t)\geq\inf_{y\in M}h_{y,\varphi(y)}(x,t).
	\]
	This finishes the proof of the proposition.
\end{proof}

Similarly, for the forward semigroup $\{T^+_t\}_{t\geq 0}$, we have
\[
T^{+}_t\varphi(x)=\sup_{y\in M}h^{y,\varphi(y)}(x,t).
\]

The properties of $T^{+}_t$ can be obtained in a similar manner to those of $T^-_t$ and thus will be omitted.
The semigroup $\{T^+\}_{t\geq 0}$ will be used to tackle other related problems in our forthcoming work.
In the following we will only study the properties of $T^{-}_t$ and denote $T^-_t$ by $T_t$ for brevity.

By Proposition \ref{pr4.1}, we now show the semigroup property of $\{T_t\}_{t\geq 0}$.

\begin{proposition}\label{pr4.3}
	$\{T_t\}_{t\geq 0}$ is a one-parameter semigroup of operators.
\end{proposition}

\begin{proof}
	It is easy to check that $T_0=I$, where $I$ denotes unit operator. We only need to show that $T_{s+t}=T_t\circ T_s$ for all $t$, $s>0$.

	In view of the definition of $T_t$ and Proposition \ref{Markov}, we have
	\begin{align}\label{4-2}
	T_sh_{x_0,u_0}(x,t)=\inf_{y\in M}h_{y,h_{x_0,u_0}(y,t)}(x,s)=h_{x_0,u_0}(x,t+s),\quad \forall x_0,\, x\in M,\, \forall u_0\in\mathbf{R},\, \forall t, s>0.
	\end{align}
	By Proposition \ref{pr4.1} and (\ref{4-2}), we have
	\[
	T_{t+s}\varphi(x)=\inf_{y\in M}h_{y,\varphi(y)}(x,t+s)=\inf_{y\in M}(\inf_{z\in M}h_{z,h_{y,\varphi(y)}(z,s)}(x,t))=\inf_{z\in M}(\inf_{y\in M}h_{z,h_{y,\varphi(y)}(z,s)}(x,t)).
	\]
	On the other hand,
	\begin{align*}
	(T_t\circ T_s\varphi)(x)&=T_t(T_s\varphi)(x)=\inf_{z\in M}h_{z,T_s\varphi(z)}(x,t)
	=\inf_{z\in M}h_{z,\inf_{y\in M}h_{y,\varphi(y)}(z,s)}(x,t).
	\end{align*}
	It remains to verify that
	\[
	\inf_{y\in M}h_{z,h_{y,\varphi(y)}(z,s)}(x,t)=h_{z,\inf_{y\in M}h_{y,\varphi(y)}(z,s)}(x,t),\quad  \forall z\in M.
	\]
	Indeed, by the compactness of $M$, there exists $y_0$ such that $h_{y_0,\varphi(y_0)}(z,s)=\inf_{y\in M}h_{y,\varphi(y)}(z,s)$. Then
	\[
	\inf_{y\in M}h_{z,h_{y,\varphi(y)}(z,s)}(x,t)\leq h_{z,h_{y_0,\varphi(y_0)}(z,s)}(x,t)=h_{z,\inf_{y\in M}h_{y,\varphi(y)}(z,s)}(x,t).
	\]
	It follows from Proposition \ref{Mono I} that for each $y\in M$, we get
	\[
	h_{z,\inf_{y\in M}h_{y,\varphi(y)}(z,s)}(x,t)\leq h_{z,h_{y,\varphi(y)}(z,s)}(x,t),
	\]
	which implies
	\[
	h_{z,\inf_{y\in M}h_{y,\varphi(y)}(z,s)}(x,t)\leq \inf_{y\in M}h_{z,h_{y,\varphi(y)}(z,s)}(x,t).
	\]
	The proof is complete now.
\end{proof}

A direct consequence of Proposition \ref{pr4.3} is as follows.

\begin{corollary}\label{coro}
	For each $x\in M$, we have $T_{t+s}\varphi(x)=\inf_{z\in M}h_{z,T_s\varphi(z)}(x,t)$ for all $t>0$ and all $s>0$.
\end{corollary}

\begin{proposition}\label{pr4.4}
	Given $\varphi$, $\psi\in C(M,\mathbf{R})$,  we have
	\begin{itemize}
		\item [(i)] if $\psi<\varphi$, then $T_t\psi< T_t\varphi$, \quad $\forall t\geq 0$;
		\item [(ii)] the function $(x,t)\mapsto T_t\varphi(x)$ is locally Lipschitz on $M\times (0,+\infty)$.
	\end{itemize}
\end{proposition}

\begin{proof}
	First we prove (i). Since $\psi<\varphi$, we have $h_{y,\psi(y)}(x,t)<h_{y,\varphi(y)}(x,t)$ for any $y\in M$.
	For each $(x,t)$, by the compactness of $M$ there exists $y_0\in M$ such that $h_{y_0,\varphi(y_0)}(x,t)=\inf_{y\in M}h_{y,\varphi(y)}(x,t)$. It follows that
	\[
	T_t\psi(x)=\inf_{y\in M}h_{y,\psi(y)}(x,t)\leq h_{y_0,\psi(y_0)}(x,t)<h_{y_0,\varphi(y_0)}(x,t)=\inf_{y\in M}h_{y,\varphi(y)}(x,t)=T_t\varphi(x).
	\]
	(ii) is an immediate consequence of Propositions \ref{pr3.3} and \ref{pr4.1}.
\end{proof}

\subsection{Solution semigroups and  viscosity solutions}
At the end of this section, we will show the following result, which together with Propositions \ref{pr4.1} and \ref{pr4.3} implies Theorem \ref{thehj}.

\begin{proposition}\label{pr4.5}
	For any given $\varphi(x)\in C(M,\mathbf{R})$, $u(x,t):=T_t\varphi(x)$  is the unique viscosity solution of (\ref{Cau}).
\end{proposition}
By the comparison theorem (see \cite{Ba3} for instance), it yields that  the viscosity  solution of (\ref{Cau}) is unique under the assumptions (H1)-(H3).
Thus, in order to show Proposition \ref{pr4.5}, it suffices to prove Lemmas \ref{lem4.2} and \ref{lem4.3}.

\begin{lemma}\label{lem4.2}
	$u(x,t)$  is a variational solution of equation (\ref{ehj}).
\end{lemma}

\begin{proof}
	We need to show that $u$ satisfies (i) and (ii) of Definition \ref{nw}.
	(i) Let $\gamma:[t_1,t_2]\rightarrow M$ be a continuous piecewise $C^1$ curve. Let $\bar{\gamma}:[0,t_1]\rightarrow M$ be a minimizer of $u(\gamma(t_1),t_1)$. Consider a curve $\xi:[0,t_2]\rightarrow M$ defined by
	\begin{equation*}\label{xi88}
	\xi(t)=\left\{\begin{array}{ll}
	\hspace{-0.4em}\bar{\gamma}(t),& t\in [0,t_1],\\
	\hspace{-0.4em}\gamma(t),&t\in (t_1,t_2].\\
	\end{array}\right.
	\end{equation*}
	It follows that
	\begin{align*}
	&u(\gamma(t_2),t_2)-u(\gamma(t_1),t_1)\\
=& \inf_{\gamma_2(t_2)=\gamma(t_2)}\{\varphi(\gamma_2(0))+
	\int_0^{t_2}L(\gamma_2(\tau),u(\gamma_2(\tau),\tau),\dot{\gamma}_2(\tau))d\tau\}\\
	&-\inf_{\gamma_1(t_1)=\gamma(t_1)}\{\varphi(\gamma_1(0))+\int_0^{t_1}L(\gamma_1(\tau),u(\gamma_1(\tau),\tau),\dot{\gamma}_1(\tau))d\tau\}\\
	\leq& \varphi(\xi(0))+\int_0^{t_2}L(\xi(\tau),u(\xi(\tau),\tau),\dot{\xi}(\tau))d\tau-\varphi(\bar{\gamma}(0))\\
&-
	\int_0^{t_1}L(\bar{\gamma}(\tau),u(\bar{\gamma}(\tau),\tau),\dot{\bar{\gamma}}(\tau))d\tau,
	\end{align*}
	which gives rise to
	\begin{equation*}
	u(\gamma(t_2),t_2)-u(\gamma(t_1),t_1)\leq \int_{t_1}^{t_2}L(\gamma(\tau),u(\gamma(\tau),\tau),\dot{\gamma}(\tau))d\tau.
	\end{equation*}

	(ii) For each $[t_1,t_2]\subset[0,T]$ and each $x\in M$, there exists a $C^1$ minimizer $\gamma:[0,t_2]\rightarrow M$ with $\gamma(t_2)=x$ such that
	\[
	u(x,t_2)=\varphi(\gamma(0))+\int_0^tL(\gamma(\tau),u(\gamma(\tau),\tau),\dot{\gamma}(\tau))d\tau,
	\]
	which implies
	\[
	u(x,t_2)-u(\gamma(t_1),t_1)=\int_{t_1}^{t_2}L(\gamma(\tau),u(\gamma(\tau),\tau),\dot{\gamma}(\tau))d\tau.
	\]
	This completes the proof of Lemma \ref{lem4.2}.
\end{proof}

\begin{lemma}\label{lem4.3}
	A variational solution of equation (\ref{ehj}) is also a viscosity solution.
\end{lemma}

Since the proof of Lemma \ref{lem4.3} is only slightly different from the one
of Proposition 7.27 in \cite{Fat08},
we omit it here for brevity.



\section{Application II: Ergodic problem for $H(x,u,u_x)=c$}
The goal of this part is to prove Theorem \ref{thshj}. For each $c\in \mathbf{R}$, since $L+c$ satisfies all the assumptions imposed on $L$, then  the implicit variational principle and all the results established for $L$ in this paper are still correct for $L+c$.

Denote by $h^c_{x_0,u_0}(x,t)$, $h_c^{x_0,u_0}(x,t)$ and $T^c_t\varphi(x)$ the forward implicit action function, the backward implicit action function and the solution semigroup associated with $L+c$, respectively.

\subsection{Implicit action function associated with $L+c$}
We give two properties of the function $c\mapsto h^c_{x_0,u_0}(x,t)$.

\begin{proposition}[Monotonicity property III]\label{pr5.1}
	Given $x_0\in M$, $u_0\in\mathbf{R}$ and $c_1$, $c_2\in\mathbf{R}$, if $c_1<c_2$, then $h^{c_1}_{x_0,u_0}(x,t)<h^{c_2}_{x_0,u_0}(x,t)$ for all $(x,t)\in M\times (0,+\infty)$.
\end{proposition}

It is not hard to see that Proposition \ref{pr5.1} is a direct consequence of Proposition \ref{Mono II} and thus we omit the proof here.

\begin{proposition}\label{pr5.2}
	Given any $(x_0,u_0,x,t)\in\Omega_{a,b,\delta,T}$ and $c_1$, $c_2\in\mathbf{R}$, we have
	\[
	|h^{c_1}_{x_0,u_0}(x,t)-h^{c_2}_{x_0,u_0}(x,t)|\leq e^{\lambda t}t|c_1-c_2|\leq e^{\lambda T}T|c_1-c_2|.
	\]
\end{proposition}

\begin{proof}
	If $c_1\leq c_2$, then by Proposition \ref{pr5.1}, we have $h^{c_1}_{x_0,u_0}(x,t)\leq h^{c_2}_{x_0,u_0}(x,t)$. Let $\gamma_1$ be a minimizer of $h^{c_1}_{x_0,u_0}(x,t)$. Then by Proposition \ref{pr5.1} again, for any $s\in (0,t]$, we get
	\begin{equation}\label{tpro}
	h^{c_1}_{x_0,u_0}(\gamma_1(s),s)\leq h^{c_2}_{x_0,u_0}(\gamma_1(s),s).
	\end{equation}
	From (\ref{iaf}) and (L3), we have
	\begin{align*}
	&h^{c_2}_{x_0,u_0}(\gamma_1(s),s)- h^{c_1}_{x_0,u_0}(\gamma_1(s),s)\\
	\leq &(c_2-c_1)s+\int_0^sL(\gamma_1,h^{c_2}_{x_0,u_0}(\gamma_1(\tau),\tau),\dot{\gamma}_1)
	-L(\gamma_1,h^{c_1}_{x_0,u_0}(\gamma_1(\tau),\tau),\dot{\gamma}_1)d\tau\\
	 \leq &(c_2-c_1)t+\int_0^s\lambda(h^{c_2}_{x_0,u_0}(\gamma_1(\tau),\tau)- h^{c_1}_{x_0,u_0}(\gamma_1(\tau),\tau))d\tau.
	\end{align*}
	Let $F(\tau):=h^{c_2}_{x_0,u_0}(\gamma_1(\tau),\tau)- h^{c_1}_{x_0,u_0}(\gamma_1(\tau),\tau)$. It follows from (\ref{tpro}) that $F(\tau)\geq 0$ for any $\tau\in (0,t]$. Hence, we have
	\[
	F(s)\leq (c_2-c_1)t+\int_0^s\lambda F(\tau)d\tau.
	\]
	By Gronwall inequality, it yields
	\[
	F(s)\leq (c_2-c_1)te^{\lambda s}, \quad s\in[0,t].
	\]
	Thus, we have
	\[
	|h^{c_2}_{x_0,u_0}(x,t)-h^{c_1}_{x_0,u_0}(x,t)|= h^{c_2}_{x_0,u_0}(x,t)-h^{c_1}_{x_0,u_0}(x,t)\leq te^{\lambda t}(c_2-c_1)\leq Te^{\lambda T}|c_1-c_2|.
	\]
	We have shown the result for the case $c_1\leq c_2$.

	If $c_1>c_2$, then $h^{c_1}_{x_0,u_0}(x,t)\geq h^{c_2}_{x_0,u_0}(x,t)$. Let $\gamma_2$ be a minimizer of $h^{c_2}_{x_0,u_0}(x,t)$. Let $G(\tau):=h^{c_1}_{x_0,u_0}(\gamma_2(\tau),\tau)- h^{c_2}_{x_0,u_0}(\gamma_2(\tau),\tau)$. By a similar argument used in the first case, we have
	\[
	G(s)\leq (c_1-c_2)te^{\lambda s},\quad s\in[0,t].
	\]
	Thus, we have
	\[
	|h^{c_2}_{x_0,u_0}(x,t)-h^{c_1}_{x_0,u_0}(x,t)|= h^{c_1}_{x_0,u_0}(x,t)-h^{c_2}_{x_0,u_0}(x,t)\leq te^{\lambda t}(c_1-c_2)\leq Te^{\lambda T}|c_1-c_2|.
	\]

	We complete the proof of the Proposition now.
	
\end{proof}

\subsection{Proof of Theorem \ref{thshj}}
Before giving the proof of Theorem \ref{thshj}, we prove  two lemmas first.

\begin{lemma}\label{lem5.1}
	Given $\varphi\in C(M,\mathbf{R})$, there exists $\tilde{c}>0$ such that
	\begin{itemize}
		\item [(i)] $T^c_1\varphi(x)\geq\varphi(x)$, \quad $\forall c\geq\tilde{c}$, \ $\forall x\in M$;
		\item [(ii)] $T^c_1\varphi(x)\leq\varphi(x)$, \quad $\forall c\leq-\tilde{c}$, \ $\forall x\in M$.
	\end{itemize}
\end{lemma}

\begin{proof}
	(i) For any given $x\in M$ and $c\in\mathbf{R}$, there is $y_x^c\in M$ such that $T^c_1\varphi(x)=h^c_{y_x^c,\varphi(y_x^c)}(x,1)$.
	By Proposition \ref{VP}, there is a solution $(x(s),u(s),p(s))$ of the contact Hamilton's equations associated with $H-c$ with $x(0)=y_x^c$,
	$u(0)=\varphi(y_x^c)$ and $x(1)=x$. Thus,
	\[
	\dot{u}(s)=L(x(s),u(s),\dot{x}(s))+c,\quad u(1)=h^c_{y_x^c,\varphi(y_x^c)}(x,1).
	\]
	Let
	\[
	a=\inf_{\substack{(x,\dot{x})\in TM\\  u\in[-\|\varphi\|_0,\|\varphi\|_0]} }L(x,u,\dot{x}).
	\]
	Since $L$ satisfies (L2), then $a$ is well-defined. We choose $c'\in\mathbf{R}$ such that $a+c'>2\|\varphi\|_0+1$. Since $u(s)$ satisfies
	\[
	\dot{u}(s)=L(x(s),u(s),\dot{x}(s))+c',\quad u(0)=\varphi(y_x^{c'}),
	\]
	then $u(1)=h^{c'}_{y_x,\varphi(y_x^{c'})}(x,1)>\|\varphi\|_0$. Thus, we get
	\[
	T^{c'}_1\varphi(x)= h^{c'}_{y_x^{c'},\varphi(y_x^{c'})}(x,1)>\|\varphi\|_0\geq\varphi(x).
	\]
	From the arguments above, it is clear that the choice of $c'$ is independent of $x$. In view of Propositions \ref{pr5.1} and \ref{pr4.1}, we have
	\[
	T^c_1\varphi(x)\geq\varphi(x),\quad \forall x\in M,\ \forall c\geq c'.
	\]

	(ii) From Propositions \ref{pr4.1} and \ref{Mono I}, we have
	\[
	T^c_1\varphi(x)=\inf_{y\in M}h^c_{y,\varphi(y)}(x,1)\leq \inf_{y\in M}h^c_{y,\|\varphi\|_0}(x,1),\quad \forall x\in M.
	\]
	We only need to show that
	\[
	h^c_{y,\|\varphi\|_0}(x,1)\leq-\|\varphi\|_0,\quad \forall x, \ y\in M
	\]
	for $c<0$ with $-c$ large enough.

	Choose $c\in\mathbf{R}$ such that $B+c<0$, where $B:=\sup\{|L(x,0,\dot{x})|\ \big| \ \|\dot{x}\|\leq \mathrm{diam}(M)\}$. For each $x$, $y\in M$, let $\gamma:[0,1]\to M$ be a geodesic with $\gamma(0)=y$, $\gamma(1)=x$ and $\|\dot{\gamma}\|=d(x,y)$. Then by (\ref{iaf}), we have
	\[
	h^c_{y,\|\varphi\|_0}(\gamma(t),t)-h^c_{y,\|\varphi\|_0}(\gamma(s),s)\leq\int_s^t(L(\gamma,h^c_{y,\|\varphi\|_0}(\gamma(\sigma),\sigma),\dot{\gamma})+c)d\sigma,\quad \forall 0<s<t<1.
	\]
	Let $u^c(s)=h^c_{y,\|\varphi\|_0}(\gamma(s),s)$, $s\in[0,1]$. Then
	\begin{align}\label{5-20}
	\dot{u}^c(s)\leq L(\gamma(s),u^c(s),\dot{\gamma}(s))+c\leq L(\gamma(s),0,\dot{\gamma}(s))+\lambda|u^c(s)|+c\leq B+c+\lambda|u^c(s)|
	\end{align}
	and $u^c(0)=\|\varphi\|_0$. If $s_0\in[0,1]$ is a zero of $u^c$, then we get $\lambda|u^c(s)|<\frac{|B+c|}{2}$ for $s\in[0,1]$ with $|s-s_0|$ small enough. Thus, $\dot{u}^c(s)\leq B+c+\lambda|u^c(s)|\leq \frac{B+c}{2}<0$, which implies that there exists at most one zero of $u^c$ in $[0,1]$.

	We assert that  there is a constant $c''<0$ with $B+c''<0$ such that,  for all $c<c''$ there is $s_0\in[0,\frac{1}{2}]$ such that $u^c(s_0)=0$.
	If the assertion is not true, then for a constant $c\in\mathbf{R}$ with $e^\frac{\lambda}{2}\|\varphi\|_0+(B+c)\frac{1}{\lambda}(e^\frac{\lambda}{2}-1)<0$, there is $c'''<c$ such that $u^{c'''}(s)>0$, $\forall s\in[0,\frac{1}{2}]$. Hence, by (\ref{5-20}) we have
	\[
	\dot{u}^{c'''}(s)-\lambda u^{c'''}(s)\leq B+c''', \quad \forall s\in[0,\frac{1}{2}].
	\]
	Thus, we get
	\[
	\int_0^\frac{1}{2}\frac{d}{ds}(u^{c'''}(s)e^{-\lambda s})ds\leq \int_0^\frac{1}{2}e^{-\lambda s}(B+c''')ds.
	\]
	Hence,
	\[
	0<u^{c'''}(\frac{1}{2})\leq e^\frac{\lambda}{2}\|\varphi\|_0+(B+c''')\frac{1}{\lambda}(e^\frac{\lambda}{2}-1)<0,
	\]
	a contradiction.

	Therefore, there is $c''<0$ with $B+c''<0$ such that for all $c<c''$ there is $s_0\in[0,\frac{1}{2}]$ such that $u^c(s)>0$ for $s\in[0,s_0)$,  $u^c(s_0)=0$,
	$u^c(s)<0$ for $s\in(s_0,1]$. By (\ref{5-20}) we have
	\[
	\dot{u}^c(s)+\lambda u^c(s)\leq B+c, \quad \forall s\in [s_0,1].
	\]
	Thus, we get
	\[
	\int_{s_0}^1\frac{d}{ds}(e^{\lambda s}u^c(s))ds\leq\int_{s_0}^1e^{\lambda s}(B+c)ds,
	\]
	which implies
	\begin{align}\label{5-21}
	u^c(1)\leq \frac{1}{\lambda}(1-e^{\lambda(s_0-1)})(B+c).
	\end{align}
	Since $s_0\in[0,\frac{1}{2}]$, then
	\[
	u^c(1)\leq \frac{1}{\lambda}(1-e^{\lambda(s_0-1)})(B+c)\leq \frac{1}{\lambda}(1-e^{-\lambda})(B+c).
	\]
	Let
	\[
	c''''=\min\{-\frac{\|\varphi\|_0}{\frac{1}{\lambda}(1-e^\lambda)}-B,c''-1\}.
	\]
	Then by (\ref{5-21}), we have $u^{c''''}(1)\leq-\|\varphi\|_0$, i.e.,
	\[
	h^{c''''}_{y,\|\varphi\|_0}(x,1)\leq-\|\varphi\|_0.
	\]

	Note that the above arguments are independent of $x$ and $y$. Therefore, the proof is complete now.
	
\end{proof}

\begin{lemma}\label{lem5.2}
	Given $\varphi\in C(M,\mathbf{R})$, let
	\[
	c_1=\inf\{c\ | \sup_{(x,t)\in M\times[0,+\infty)} T^c_t\varphi(x)=+\infty\}, \qquad c_2=\sup\{c\ | \inf_{(x,t)\in M\times[0,+\infty)} T^c_t\varphi(x)=-\infty\}.
	\]
	Then $-\infty<c_1\leq +\infty$ and $-\infty\leq c_2<+\infty$.
\end{lemma}

\begin{proof}
	We first show $-\infty<c_1\leq +\infty$. In view of Lemma \ref{lem5.1}, for $c<0$ with $-c$ large enough, we have
	\[
	T^c_1\varphi(x)\leq \varphi(x), \quad \forall x\in M.
	\]
	By Propositions \ref{pr4.3} and \ref{pr4.4}, we have
	\begin{align}\label{5-22}
	T^c_n\varphi(x)\leq \varphi(x), \quad \forall x\in M, \ \forall n\in \mathbf{N}.
	\end{align}
	Let $A=\sup_{(x,t)\in M\times [0,1]}T^c_t\varphi(x)$. Then by Propositions \ref{pr4.3} and \ref{pr4.4}, (\ref{5-22}), we get
	\[
	T^c_t\varphi(x)=T^c_{[t]+\{t\}}\varphi(x)=T^c_{\{t\}}\circ T^c_{[t]}\varphi(x)\leq T^c_{\{t\}}\varphi(x)\leq A,\ \forall (x,t)\in M\times[0,+\infty),
	\]
	where $\{t\}$ denotes the fractional part of $t$ and $[t]$ denotes the greatest integer not greater than $t$. See Notations in the introduction section for details. In view of Propositions \ref{pr4.1} and \ref{pr5.1}, for $c''<c'$ we have
	\[
	T^{c''}_t\varphi(x)=\inf_{y\in M}h^{c''}_{y,\varphi(y)}(x,t)\leq\inf_{y\in M}h^{c'}_{y,\varphi(y)}(x,t)=T^{c'}_t\varphi(x),\ \forall (x,t)\in M\times[0,+\infty).
	\]
	Hence, for each $\bar{c}\leq c$, we have $T^{\bar{c}}_t\varphi(x)\leq T^{c}_t\varphi(x)\leq A$, $\forall (x,t)\in M\times[0,+\infty)$.
	Therefore, $c_1\neq-\infty$.

	Next we  show $-\infty\leq c_2<+\infty$. In view of Lemma \ref{lem5.1}, for $c>0$  large enough, we have
	\[
	T^c_1\varphi(x)\geq\varphi(x), \quad \forall x\in M.
	\]
	By Propositions \ref{pr4.3} and \ref{pr4.4}, we get
	\[
	T^c_n\varphi(x)\geq \varphi(x), \quad \forall x\in M, \ \forall n\in \mathbf{N}.
	\]
	Let $B=\inf_{(x,t)\in M\times [0,1]}T^c_t\varphi(x)$. Then
	\[
	T^c_t\varphi(x)=T^c_{[t]+\{t\}}\varphi(x)=T^c_{\{t\}}\circ T^c_{[t]}\varphi(x)\geq T^c_{\{t\}}\varphi(x)\geq B,\ \forall (x,t)\in M\times[0,+\infty).
	\]
	From Propositions \ref{pr4.1} and \ref{pr5.1}, for $c''>c'$ we have
	\[
	T^{c''}_t\varphi(x)=\inf_{y\in M}h^{c''}_{y,\varphi(y)}(x,t)\geq\inf_{y\in M}h^{c'}_{y,\varphi(y)}(x,t)=T^{c'}_t\varphi(x),\ \forall (x,t)\in M\times[0,+\infty),
	\]
	Hence, for each $\bar{c}\geq c$, we have $T^{\bar{c}}_t\varphi(x)\geq T^{c}_t\varphi(x)\geq B$, $\forall (x,t)\in M\times[0,+\infty)$. Therefore, $c_2\neq+\infty$.
	
\end{proof}

\noindent\emph{Proof of Theorem \ref{thshj}}.
Given $\varphi\in C(M,\mathbf{R})$, let $c_1$, $c_2$ be as in Lemma \ref{lem5.2}.
In view of Lemma \ref{lem5.2}, we prove the theorem in the following three cases:\\[1mm]
\noindent \textbf{Case I}: $c_1\in \mathbf{R}$.\\[1mm]
\textbf{Case II}: $c_1=+\infty$, $c_2=-\infty$.\\[1mm]
\textbf{Case III}: $c_1=+\infty$, $c_2\in \mathbf{R}$.

\vskip0.1cm
Our plan is:  1) to show that there is $c\in\mathbf{R}$ such that $T^c_t\varphi(x)$ is uniformly bounded with respect to $t>1$ in all three cases; 2) to show that $c$ is the constant for which equation (\ref{shj}) has viscosity solutions.

\vskip.2cm

\noindent\textbf{Step 1}. We first show that there exists $c\in\mathbf{R}$ such that $T^c_t\varphi(x)$ is uniformly bounded with respect to $t>1$ in all three cases.

For Case II, by the definitions of $c_1$ and $c_2$, it is straightforward to see that for each $c
\in\mathbf{R}$, $T^c_t\varphi(x)$ is uniformly bounded with respect to $t>0$. For Case III, by the definitions of $c_1$ and $c_2$,  the uniform boundedness of $T^c_t\varphi(x)$ holds for each $c>c_2$.

For Case I, we will show that there is a constant $c$ such that $T^c_t\varphi(x)$ is uniformly bounded with respect to $t>1$, when $c_1\in \mathbf{R}$. In this case, there are at most three possibilities:\\[2mm]
(i) for each $t>0$, there is $x_t\in M$ such that $T^{c_1}_t\varphi(x_t)=\varphi(x_t)$;\\[2mm]
(ii) there exists $t_0>0$  such that $T^{c_1}_{t_0}\varphi(x)<\varphi(x)$, $\forall x\in M$; \\[2mm]
(iii) there exists $t_0>0$  such that $T^{c_1}_{t_0}\varphi(x)>\varphi(x)$, $\forall x\in M$.

\vskip0.3cm

(i) By Corollary \ref{coro}, we have
\begin{align*}
T^{c_1}_t\varphi(x)=\inf_{z\in M}h^{c_1}_{z,T^{c_1}_{t-1}\varphi(z)}(x,1)\leq h^{c_1}_{x_{t-1},T^{c_1}_{t-1}\varphi(x_{t-1})}(x,1)
=h^{c_1}_{x_{t-1},\varphi(x_{t-1})}(x,1).
\end{align*}
Note that the function $(x_0,u_0,x)\mapsto h^{c_1}_{x_0,u_0}(x,1)$ is continuous on $M\times[-\|\varphi\|_0,\|\varphi\|_0]\times M$.
Thus, $T^{c_1}_t\varphi(x)$ is uniformly bounded from above with respect to $t>1$.

On the other hand, for each $x\in M$, by Proposition \ref{pr4.3} we get
\[
h^{c_1}_{x,T^{c_1}_t\varphi(x)}(x_{t+1},1)\geq\inf_{z\in M} h^{c_1}_{z,T^{c_1}_t\varphi(z)}(x_{t+1},1)=T^{c_1}_{t+1}\varphi(x_{t+1})=\varphi(x_{t+1})\geq-\|\varphi\|_0.
\]
By Proposition \ref{pr3.4}, one can find $u^*\in \mathbf{R}$ such that
\[h_{x,u^*}^{c_1}(x_{t+1},1)=-\|\varphi\|_0.\]
From Proposition \ref{relation}, we get
\[h^{x_{t+1},-\|\varphi\|_0}_{c_1}(x,1)=u^*,\]
which shows
\[
h^{c_1}_{x,h_{c_1}^{x_{t+1},-\|\varphi\|_0}(x,1)}(x_{t+1},1)=-\|\varphi\|_0,\quad \forall x\in M.
\]
In view of Proposition \ref{Mono I}, we get
\[
h_{c_1}^{x_{t+1},-\|\varphi\|_0}(x,1)\leq T^{c_1}_t\varphi(x), \quad \forall x\in M.
\]
Since the function $(x_0,x)\mapsto h_{c_1}^{x_0,-\|\varphi\|_0}(x,1)$ is continuous on $M\times M$, then $T^{c_1}_t\varphi(x)$ is uniformly bounded with respect to $t>0$.

(ii) Now we will show the possibility (ii) does not exist. Otherwise, there is $t_0>0$  such that $T^{c_1}_{t_0}\varphi(x)<\varphi(x)$, $\forall x\in M$. By Propositions \ref{pr4.1}, \ref{pr5.2} and the compactness of $M$, then there exists $\varepsilon_0>0$ such that
\begin{align}\label{5-1}
T^{c_1+\varepsilon_0}_{t_0}\varphi(x)<\varphi(x),\quad \forall x\in M.
\end{align}
Let $A=\sup_{(x,t)\in M\times[0,t_0]}T^{c_1+\varepsilon_0}_t\varphi(x)$. Let
$v^{c_1+\varepsilon_0}(x,t)=T^{c_1+\varepsilon_0}_t\circ T^{c_1+\varepsilon_0}_{t_0}\varphi(x)$. Then from (\ref{5-1}) and Proposition \ref{pr4.4}, we have
\begin{align}\label{5-2}
v^{c_1+\varepsilon_0}(x,s)=T^{c_1+\varepsilon_0}_s\circ T^{c_1+\varepsilon_0}_{t_0}\varphi(x)<T^{c_1+\varepsilon_0}_s\varphi(x),\quad \forall (x,s)\in M\times (0,+\infty).
\end{align}
Note that $T^{c_1+\varepsilon_0}_s\varphi(x)=v^{c_1+\varepsilon_0}(x,s-t_0)$ for $\forall (x,s)\in M\times [t_0,+\infty)$. By (\ref{5-2}) we have
\begin{align}\label{5-3}
v^{c_1+\varepsilon_0}(x,s-t_0)<T^{c_1+\varepsilon_0}_{s-t_0}\varphi(x)\leq A,\quad \forall (x,s)\in M\times[t_0,2t_0].
\end{align}
Hence, for each $(x,s)\in M\times[2t_0,3t_0]$, by (\ref{5-3}) we have
\[
T^{c_1+\varepsilon_0}_s\varphi(x)=v^{c_1+\varepsilon_0}(x,s-t_0)<T^{c_1+\varepsilon_0}_{s-t_0}\varphi(x)\leq A.
\]
Therefore,
\[
T^{c_1+\varepsilon_0}_s\varphi(x)\leq A,\quad \forall(x,s)\in M\times[0,+\infty).
\]
Recall that $c_1=\inf\{c\ | \sup_{(x,t)\in M\times[0,+\infty)} T^c_t\varphi(x)=+\infty\}$. There exists $\bar{c}<c_1+\varepsilon_0$ such that
\begin{align}\label{5-4}
\sup_{(x,t)\in M\times[0,+\infty)} T^{\bar{c}}_t\varphi(x)=+\infty.
\end{align}
From Proposition \ref{pr5.1}, we get
\[
h^{\bar{c}}_{y,\varphi(y)}(x,t)<h^{c_1+\varepsilon_0}_{y,\varphi(y)}(x,t),\quad \forall x,\ y\in M,\ \forall t>0.
\]
From Proposition \ref{pr4.1}, we have
\[
T^{\bar{c}}_t\varphi(x)=\inf_{y\in M}h^{\bar{c}}_{y,\varphi(y)}(x,t)\leq\inf_{y\in M}h^{c_1+\varepsilon_0}_{y,\varphi(y)}(x,t)=T^{c_1+\varepsilon_0}_t\varphi(x)\leq A
\]
for all $(x,t)\in M\times[0,+\infty)$, which contradicts (\ref{5-4}).

(iii) If there exists $t_0>0$ such that $T^{c_1}_{t_0}\varphi(x)>\varphi(x)$ for all $x\in M$, then there is $\varepsilon_0>0$ such that for each $\varepsilon\in(0,\varepsilon_0)$, we have
\[
T^{c_1-\varepsilon}_{t_0}\varphi(x)>\varphi(x), \quad \forall x\in M.
\]

By the definition of $c_1$, $T^{c_1-\varepsilon}_t\varphi(x)$ is uniformly bounded from above with respect to $t\geq0$, i.e.,
\[
\sup_{(x,t)\in M\times[0,+\infty)}T^{c_1-\varepsilon}_t\varphi(x)<+\infty.
\]
It is sufficient to show that $T^{c_1-\varepsilon}_t\varphi(x)$ is uniformly bounded from below with respect to $t\geq0$. Let $B=\min_{(x,t)\in M\times[0,t_0]}T^{c_1-\varepsilon}_t\varphi(x)$. Then
$T^{c_1-\varepsilon}_t\varphi(x)\geq B$ for all $(x,t)\in M\times [0,t_0]$. Let $w^{c_1-\varepsilon}(x,t)=T^{c_1-\varepsilon}_t\circ T^{c_1-\varepsilon}_{t_0}\varphi(x)$. Then
\[
w^{c_1-\varepsilon}(x,s)=T^{c_1-\varepsilon}_s\circ T^{c_1-\varepsilon}_{t_0}\varphi(x)>T^{c_1-\varepsilon}_s\varphi(x), \quad \forall (x,s)\in M\times[0,+\infty).
\]
Note that $T^{c_1-\varepsilon}_t\varphi(x)=w^{c_1-\varepsilon}(x,t-t_0)$ for all $t\geq t_0$. Thus,
\[
T^{c_1-\varepsilon}_s\varphi(x)=w^{c_1-\varepsilon}(x,s-t_0)>T^{c_1-\varepsilon}_{s-t_0}\varphi(x)\geq B,\quad (x,s)\in M\times [t_0,2t_0],
\]
which implies that
\[
T^{c_1-\varepsilon}_s\varphi(x)\geq B,\quad \forall(x,s)\in M\times [0,+\infty).
\]

\vskip.2cm

\noindent\textbf{Step 2}.
In the first step, we have shown the existence of the constant $c\in\mathbf{R}$ for which $T^c_t\varphi(x)$ is uniformly bounded with respect to $t>1$. We are now in a position to show that $\varphi_\infty(x):=\liminf_{t\to+\infty}T^c_t\varphi(x)$ is a viscosity solution to equation (\ref{shj}).

Let $K>0$ be a constant such that $|T^c_t\varphi(x)|\leq K$ for all $x\in M$ and all $t>1$. Since
\[
|T^c_t\varphi(x)-T^c_t\varphi(y)|\leq\sup_{z\in M}|h^c_{z,T^c_{t-1}\varphi(z)}(x,1)-h^c_{z,T^c_{t-1}\varphi(z)}(y,1)|,\quad \forall t>2,\ \forall x,\ y\in M,
\]
then by Proposition \ref{pr3.3}, $h^c_{\cdot,\cdot}(\cdot,1)$ is Lipschitz on $M\times[-K,K]\times M$ with Lipschitz constant $l_1$. Thus, we get
\[
|T^c_t\varphi(x)-T^c_t\varphi(y)|\leq l_1d(x,y).
\]
Hence, $\{T^c_t\varphi(x)\}_{t>2}$ is uniformly bounded and equi-Lipschitz on $M$.

Let $\varphi_\infty(x):=\liminf_{t\to+\infty}T^c_t\varphi(x)$, $\forall x\in M$. Then, from the uniform boundedness of $\{T^c_t\varphi(x)\}_{t>2}$, it is clear that $\varphi_\infty(x)$ is well-defined. By definition, we have
\[
\lim_{t\to+\infty}\inf_{s\geq t}T^c_s\varphi(x)=\varphi_\infty(x),\quad \forall x\in M.
\]
Since
\[
|\inf_{s\geq t}T^c_s\varphi(x)-\inf_{s\geq t}T^c_s\varphi(y)|\leq\sup_{s\geq t}|T^c_s\varphi(x)-T^c_s\varphi(y)|\leq l_1d(x,y),
\quad \forall t>2,
\]
then
\begin{align}\label{5-23}
\lim_{t\to+\infty}\inf_{s\geq t}T^c_s\varphi(x)=\varphi_\infty(x)
\end{align}
uniformly on $x\in M$.

We assert that $\varphi_\infty$ is a fixed point of $\{T^c_t\}$. In fact, for each $t>0$, we have
\begin{align*}
\varphi_\infty(x)=\lim_{\sigma\to+\infty}\inf_{s\geq \sigma}T^c_{s+t}\varphi(x)=\lim_{\sigma\to+\infty}\inf_{s\geq \sigma}\inf_{z\in M}h^c_{z,T^c_s\varphi(z)}(x,t)=\lim_{\sigma\to+\infty}\inf_{z\in M}h^c_{z,\inf_{s\geq \sigma}T^c_s\varphi(z)}(x,t),
\end{align*}
where the second equality is a consequence of Corollary \ref{coro} and the last one follows from Proposition \ref{Mono I}.
Since
\begin{align*}
|\inf_{z\in M}h^c_{z,\inf_{s\geq \sigma}T^c_s\varphi(z)}(x,t)-T^c_t\varphi_\infty(x)|&=|\inf_{z\in M}h^c_{z,\inf_{s\geq \sigma}T^c_s\varphi(z)}(x,t)-\inf_{z\in M}h^c_{z,\varphi_\infty(z)}(x,t)|\\
&\leq\sup_{z\in M}|h^c_{z,\inf_{s\geq \sigma}T^c_s\varphi(z)}(x,t)-h^c_{z,\varphi_\infty(z)}(x,t)|\\
&\leq l_t\|\inf_{s\geq \sigma}T^c_s\varphi-\varphi_\infty\|_0
\end{align*}
for $\sigma>0$ large enough and (\ref{5-23}), then
\[
\varphi_\infty(x)=T^c_t\varphi_\infty(x),\quad \forall x\in M,
\]
where $l_t$ is the Lipschitz constant of the function $(x_0,u_0,x)\mapsto h^c_{x_0,u_0}(x,t)$ on $M\times[-K-\|\varphi_\infty\|_0,K+\|\varphi_\infty\|_0]\times M$. Thus, $T^c_t\varphi_\infty=\varphi_\infty$ for all $t>0$.

By Proposition \ref{pr4.5}, the function $(x,t)\mapsto T^c_t\varphi_\infty(x)$ is a viscosity solution of
\[
w_t+H(x,w,w_x)=c.
\]
Since $\varphi_\infty=T^c_t\varphi_\infty$ for all $t>0$, then $\varphi_\infty$ is a viscosity solution of
\[
H(x,u,u_x)=c.
\]
The proof of Theorem \ref{thshj} is complete.  \hfill$\Box$


\section{Appendix: Proof of Lemma \ref{lem2.1} }

%
%

Given  $a$, $b$, $\delta$, $T\in\mathbf{R}$ with  $a<b$, $0<\delta<T$, recall
\[
\Omega_{a,b,\delta,T}=M\times [a,b]\times M\times [\delta,T].
\]
Let
\[
k=\frac{\mathrm{diam}(M)}{\delta},\quad A=\sup_{\|\dot{x}\|\leq k}L(x,0,\dot{x}),\quad B=\inf_{(x,\dot{x})\in TM}L(x,0,\dot{x}).
\]

\begin{lemma}\label{lemA}
	There is a constant $C_{a,b,\delta,T}>0$ such that
	\[
	|h_{x_0,u_0}(x,t)|\leq C_{a,b,\delta,T}, \quad \forall (x_0,u_0,x,t)\in\Omega_{a,b,\delta,T},
	\]
	where the constant $C_{a,b,\delta,T}$ depends only on $a$, $b$, $\delta$ and $T$.
\end{lemma}

\begin{proof}
	\noindent \textbf{Boundedness from below}. Given any $(x_0,u_0,x,t)\in\Omega_{a,b,\delta,T}$, let $\gamma:[0,t]\to M$ be a minimizer of $h_{x_0,u_0}(x,t)$ and
	$u(s)=h_{x_0,u_0}(\gamma(s),s)$, $s\in[0,t]$. Then $u(t)=h_{x_0,u_0}(x,t)$. We need to show that $u(t)$ is bounded from below by a constant
	which depends only on $a$, $b$, $\delta$ and $T$. There are three possibilities:\\[2mm]
	(i) $u(t)>0$;\\[2mm]
	(ii) $u(s)<0$, $\forall s\in [0,t]$; \\[2mm]
	(iii) there exists $s_0\in[0,t]$  such that $u(s_0)=0$ and $u(s)\leq 0$, $\forall s\in[s_0,t]$.
	
	\vskip0.3cm
	
	(i) $u(t)$ is bounded from below by 0. Thus, we only need to deal with possibilities (ii) and (iii).

	(ii) Note that $u$ satisfies
	\[
	\dot{u}(s)=L(\gamma(s),u(s),\dot{\gamma}(s))\geq  L(\gamma(s),0,\dot{\gamma}(s))+\lambda u(s)\geq B+\lambda u(s),\quad s\in[0,t]
	\]
	and $u(0)=u_0$. Consider the solution $w_1(s)$ of the Cauchy problem
	\[
	\dot{w}_1(s)=B+\lambda w_1(s),\quad w_1(0)=u_0.
	\]
	It is easy to see that $w_1(s)=u_0e^{\lambda s}+\frac{B}{\lambda}(e^{\lambda s}-1)$ and
	\[
	u(t)\geq w_1(t)=u_0e^{\lambda t}+\frac{B}{\lambda}(e^{\lambda t}-1)\geq -|a|e^{\lambda T}-\frac{|B|}{\lambda}(e^{\lambda T}-1).
	\]

	(iii) In this case, $\dot{u}(s)\geq B+\lambda u(s)$ for $s\in[s_0,t]$ and $u(s_0)=0$. Let $w_2(s)$ be the solution of the Cauchy problem
	\[
	\dot{w}_2(s)=B+\lambda w_2(s),\quad w_2(s_0)=0.
	\]
	Then $w_2(s)=\frac{B}{\lambda}(e^{\lambda(s-s_0)}-1)$. Thus, we have
	\[
	u(t)\geq w_2(t)=\frac{B}{\lambda}(e^{\lambda(t-s_0)}-1)\geq-\frac{|B|}{\lambda}(e^{\lambda T}-1).
	\]

	Therefore, we get
	\[
	h_{x_0,u_0}(x,t)\geq -|a|e^{\lambda T}-\frac{|B|}{\lambda}(e^{\lambda T}-1).
	\]

	\vskip.2cm
	
	\noindent \textbf{Boundedness from above}.
	Given any $(x_0,u_0,x,t)\in\Omega_{a,b,\delta,T}$, let $\alpha:[0,t]\to M$ be a geodesic between $x_0$ and $x$ with $\|\dot{\alpha}\|=\frac{d(x_0,x)}{t}\leq\frac{\mathrm{diam}(M)}{\delta}=k$. Let $v(s)=h_{x_0,u_0}(\alpha(s),s)$, $s\in[0,t]$. Then
	$v(t)=h_{x_0,u_0}(x,t)$ and $v(0)=u_0$ by Lemma 3.1 and Lemma 3.2 in \cite{WWY}. By (\ref{iaf}) we have
	\[
	v(s_2)-v(s_1)\leq \int_{s_1}^{s_2}L(\alpha(s),v(s),\dot{\alpha}(s))ds,\quad \forall 0\leq s_1<s_2\leq t.
	\]
	Thus, we get
	\[
	\dot{v}(s)\leq L(\alpha(s),v(s),\dot{\alpha}(s))\leq L(\alpha(s),0,\dot{\alpha}(s))+\lambda|v(s)|.
	\]
	We need to show that $v(t)$ is bounded from above by a constant
	which depends only on $a$, $b$, $\delta$ and $T$. There are three possibilities:\\[2mm]
	(1) $v(t)<0$;\\[2mm]
	(2) $v(s)>0$, $\forall s\in [0,t]$; \\[2mm]
	(3) there exists $s'\in[0,t]$  such that $v(s')=0$ and $v(s)\geq 0$, $\forall s\in[s',t]$.
	
	\vskip0.3cm

	(1) $v(t)$ is bounded from above by 0. Thus, we only need to deal with possibilities (2) and (3).

	(2) Since $v(s)>0$ for all $s\in[0,t]$, then
	\[
	\dot{v}(s)\leq L(\alpha(s),0,\dot{\alpha}(s))+\lambda|v(s)|\leq A+\lambda v(s)
	\]
	and $v(0)=u_0$.  Let $w_3(s)$ be the solution of the Cauchy problem
	\[
	\dot{w}_3(s)=A+\lambda w_3(s),\quad w_3(0)=u_0.
	\]
	Then $w_3(s)=u_0e^{\lambda s}+\frac{A}{\lambda}(e^{\lambda s}-1)$. Thus, we get
	\[
	v(t)\leq w_3(t)=u_0e^{\lambda t}+\frac{A}{\lambda}(e^{\lambda t}-1)=|b|e^{\lambda T}+\frac{|A|}{\lambda}(e^{\lambda T}-1).
	\]

	(3) In this case, $\dot{v}(s)\leq A+\lambda v(s)$ for $s\in[s',t]$ and $v(s')=0$. Let $w_4(s)$ be the solution of the Cauchy problem
	\[
	\dot{w}_4(s)=A+\lambda w_4(s),\quad w_4(s')=0.
	\]
	Then $w_4(s)=\frac{A}{\lambda}(e^{\lambda(s-s')}-1)$. Thus, we have
	\[
	v(t)\leq w_4(t)=\frac{A}{\lambda}(e^{\lambda(t-s')}-1)\leq\frac{|A|}{\lambda}(e^{\lambda T}-1).
	\]

	Hence, we have
	\[
	h_{x_0,u_0}(x,t)\leq |b|e^{\lambda T}+\frac{|A|}{\lambda}(e^{\lambda T}-1).
	\]
	
\end{proof}

\begin{lemma}\label{lemB}
	Given any $(x_0,u_0,x,t)\in\Omega_{a,b,\delta,T}$, let $\gamma:[0,t]\to M$ be a minimizer of $h_{x_0,u_0}(x,t)$. Then
	\[
	|h_{x_0,u_0}(\gamma(s),s)|\leq K_{a,b,\delta,T},\quad \forall s\in[0,t],
	\]
	where $K_{a,b,\delta,T}$ is a positive constant which depends only on $a$, $b$, $\delta$ and $T$.
\end{lemma}

\begin{proof}\noindent \textbf{Boundedness from below}.
	By similar arguments used in the first part of the proof of Lemma \ref{lemA}, one can show that
	$h_{x_0,u_0}(\gamma(s),s)$ is bounded from below by a constant which depends only on $a$ and $T$.
	We omit the details for brevity.
	
	\vskip.2cm
	
	\noindent \textbf{Boundedness from above}.
	We only need to show that there exists a constant  $K_{a,b,\delta,T}>0$ such that
	\[
	h_{x_0,u_0}(\gamma(s),s)\leq K_{a,b,\delta,T},\quad \forall s\in[0,t].
	\]
	Let $u(s)=h_{x_0,u_0}(\gamma(s),s)$, $s\in[0,t]$ and $u_e=h_{x_0,u_0}(x,t)$. Let $C_{a,b,\delta,T}$ be as in Lemma \ref{lemA}. Then $|u_e|\leq C_{a,b,\delta,T}$ and there are two possibilities:\\[2mm]
	(1) $u_e>0$;\\[2mm]
	(2) $u_e\leq 0$. \\[2mm]

	(1) We assert that
	\[
	u(s)\leq \frac{|B|}{\lambda}+(C_{a,b,\delta,T}+1+\frac{|B|}{\lambda})e^{\lambda T},\quad \forall s\in[0,t].
	\]
	Otherwise, there would be $s_1\in[0,t]$ such that $u(s_1)>\frac{|B|}{\lambda}+(C_{a,b,\delta,T}+1+\frac{|B|}{\lambda})e^{\lambda T}$. Then there is $s_2\in[0,t]$ such that $u(s_2)=u_e$ and
	\[
	u(s)>u_e>0,\quad \forall s\in[s_1,s_2].
	\]
	Note that
	\[
	\dot{u}(s)=L(\gamma(s),u(s),\dot{\gamma}(s))\geq  L(\gamma(s),0,\dot{\gamma}(s))-\lambda |u(s)|\geq B-\lambda u(s),\quad s\in[s_1,s_2].
	\]
	Let $w(s)$ be the solution of the Cauchy problem
	\[
	\dot{w}(s)=B-\lambda w(s),\quad w(s_1)=u(s_1).
	\]
	Then $w(s)=e^{-\lambda (s-s_1)}\big(u(s_1)-\frac{B}{\lambda}\big)+\frac{B}{\lambda}$. Thus, we get
	\[
	u(s_2)\geq w(s_2)=e^{-\lambda (s_2-s_1)}\big(u(s_1)-\frac{B}{\lambda}\big)+\frac{B}{\lambda},
	\]
	which together with $u(s_1)>\frac{|B|}{\lambda}+(C_{a,b,\delta,T}+1+\frac{|B|}{\lambda})e^{\lambda T}$ implies
	\[
	u(s_2)>u_e+1,
	\]
	a contradiction. Hence, the assertion is true.

	(2) In this case, we assert that
	\[
	u(s)\leq \frac{|B|}{\lambda}+(2+\frac{|B|}{\lambda})e^{\lambda T}.
	\]
	If the assertion is not true, there would be $s_1$, $s_2\in[0,t]$ such that
	\[
	u(s_1)> \frac{|B|}{\lambda}+(2+\frac{|B|}{\lambda})e^{\lambda T},\quad u(s_2)=1
	\]
	and $u(s)\geq 1$ for all $s\in[s_1,s_2]$.

	Note that
	\[
	\dot{u}(s)\geq B-\lambda u(s),\quad s\in[s_1,s_2].
	\]
	Let $w'(s)$ be the solution of the Cauchy problem
	\[
	\dot{w}'(s)=B-\lambda w'(s),\quad w'(s_1)=u(s_1).
	\]
	Then $w'(s)=e^{-\lambda (s-s_1)}\big(u(s_1)-\frac{B}{\lambda}\big)+\frac{B}{\lambda}$. Thus, in view of $u(s_1)> \frac{|B|}{\lambda}+(2+\frac{|B|}{\lambda})e^{\lambda T}$ and $u(s_2)=1$, we have
	\[
	u(s_2)\geq w'(s_2)=e^{-\lambda (s_2-s_1)}\big(u(s_1)-\frac{B}{\lambda}\big)+\frac{B}{\lambda}>1,
	\]
	a contradiction.
	
\end{proof}

\vskip.2cm

\noindent\emph{Proof of Lemma \ref{lem2.1}}
Let $u(s)=h_{x_0,u_0}(\gamma(s),s)$, $s\in[0,t]$. Let $K_{a,b,\delta,T}$ be as in Lemma \ref{lemB}.
By Lemma \ref{lemB}, we have
\[
|h_{x_0,u_0}(\gamma(s),s)|\leq K_{a,b,\delta,T},\quad \forall s\in[0,t].
\]
Then from (L2) there is a constant $D:=D_{a,b,\delta,T}\in\mathbf{R}$ such that
\[
L(\gamma(s),u(s),\dot{\gamma}(s))\geq \|\dot{\gamma}(s)\|+D,\quad \forall s\in[0,t].
\]
Choose $Q:=Q_{a,b,\delta,T}>0$ such that
\[
a+Q\delta-|D|T>K_{a,b,\delta,T}.
\]

We assert that there is $s_0\in[0,t]$ such that $\|\dot{\gamma}(s_0)\|\leq Q$.
If this assertion is not true, then $\|\dot{\gamma}(s)\|> Q$, $\forall s\in[0,t]$.
Since
\[
\dot{u}(s)=L(\gamma(s),u(s),\dot{\gamma}(s))\geq \|\dot{\gamma}(s)\|+D,
\]
then
\[
\int_0^t\dot{u}(s)ds\geq\int_0^t\|\dot{\gamma}(s)\|ds+Dt.
\]
Thus, we get
\[
u(t)\geq u_0+Qt+Dt\geq a+Qt+Dt>a+Q\delta-|D|T>K_{a,b,\delta,T},
\]
a contradiction.

Recall that $(\gamma(s),u(s),p(s))$ satisfies equations (\ref{che}), then
\[
\frac{dH}{ds}(\gamma(s),u(s),p(s))=-H(\gamma(s),u(s),p(s))\frac{\partial H}{\partial u}(\gamma(s),u(s),p(s)).
\]
By (H3), we get

\begin{align}\label{6-1}
|H(\gamma(s),u(s),p(s))|\leq |H(\gamma(s_0),u(s_0),p(s_0))|e^{\lambda T}.
\end{align}
Since $\|\dot{\gamma}(s_0)\|\leq Q$, $|u(s_0)|\leq K_{a,b,\delta,T}$ and $p(s_0)=\frac{\partial L}{\partial \dot{x}}(\gamma(s_0),u(s_0),\dot{\gamma}(s_0))$, then $p(s_0)$ is bounded by a constant which depends only on $a$, $b$, $\delta$ and $T$.
Then by (\ref{6-1}) and (H2), we get
\[
\|p(s)\|\leq E_{a,b,\delta,T},
\]
where $E_{a,b,\delta,T}$ is a positive constant which depends only on $a$, $b$, $\delta$ and $T$.

\hfill$\Box$


\vskip 1cm

\noindent {\bf Acknowledgements:}
The authors would like to thank the referees for their valuable comments and suggestions, which helped to improve the presentation of this paper.
Kaizhi Wang is supported by NSFC Grant No. 11371167.
Lin Wang is supported by NSFC Grant No. 11631006, 11401107.
Jun Yan is supported by NSFC Grant No. 11631006, 11325103 and National Basic Research Program of China (973 Program) Grant No. 2013CB834102.

\end{document}